\documentclass[12pt]{amsart}
\usepackage{amssymb,latexsym, amscd,pb-diagram}
\usepackage{sseq}
\usepackage[all]{xy}
\usepackage[square, numbers]{natbib}
\usepackage{graphicx}
\usepackage{mathrsfs}
\usepackage{color}
\usepackage{amsmath}
\usepackage{multirow}
\usepackage{hyperref}
\usepackage{todonotes}

\usepackage{blindtext}
\usepackage{geometry}
\usepackage[utf8]{inputenc}
\usepackage[T1]{fontenc}
\usepackage{amsfonts,  amsthm}
\usepackage{mathtools}
\usepackage{txfonts}
\usepackage{mathrsfs}
\usepackage{ textcomp }

 \newtheorem{theorem}{Theorem}[section]
 \newtheorem{corollary}[theorem]{Corollary}
  
 \newtheorem{lemma}[theorem]{Lemma}
 \newtheorem{proposition}[theorem]{Proposition}
\newtheorem*{theorem*}{Theorem}
 \theoremstyle{definition}
 \newtheorem{definition}[theorem]{Definition}
\theoremstyle{definition}
\newtheorem*{definition*}{Definition}
 \theoremstyle{definition}

 \newtheorem{remark}[theorem]{Remark}
 
 \numberwithin{equation}{section}
\newtheorem*{conjecture*}{Conjecture}

\newcounter{commentcounter}

\begin{document}

\title[Equivariant Milnor map]{Equivariant Milnor map}



\author{Mathilda Campillo}
\address{Mathematics Department, Stony Brook University, Stony Brook NY, 11794-3651, USA.}
\email{mathilda.campillo@stonybrook.edu}

\author{Yuanxin Guan}
\address{
School of Mathematical Sciences, Fudan University, 220 Handan Rd., Yangpu District, Shanghai, 200433, People’s Republic of China}
\email{yxguan21@m.fudan.edu.cn}

\author{Zhi L\"u}
\address{
School of Mathematical Sciences, Fudan University, 220 Handan Rd., Yangpu District, Shanghai, 200433, People’s Republic of China}
\email{zlu@fudan.edu.cn}

\author{Bernardo Uribe}
\address{Departamento de Matem\'{a}ticas, Universidad Nacional de Colombia, 
 Carrera 65 Nro. 59A - 110, Medell\'{i}n,
050034, Colombia.}
\email{buribej@unal.edu.co, buribe@gmail.com}


\subjclass[2020]{
(primary) 57R77, 	57R80, (secondary) 	55N22}
\date{\today}
\keywords{Equivariant unitary bordism, magnetic unitary bordism, unitary bordism, unoriented bordism.}
\begin{abstract}
The Milnor map is the homomorphism from the unitary bordism ring to the unoriented bordism ring, halving the dimension, that maps the unitary bordism classes of the complex Milnor hypersurfaces to the unoriented bordism classes of their real points. In this work, we propose to generalize this construction to the equivariant setup and we show the existence of such a map for the equivariant unitary groups of the circle and the cyclic group of order two. Furthermore, we relate the kernel of these Milnor maps to the magnetic unitary equivariant bordism groups of free conjugations.
\end{abstract}

\maketitle

\tableofcontents

\section*{Introduction} 

A remarkable result of Conner and Floyd  states that any smooth complex algebraic variety $V_\mathbb{C}$, whose real points $V_\mathbb{R}$ are also smooth,  is cobordant to the product $V_\mathbb{R} \times V_\mathbb{R}$
in the unoriented bordism group
{\cite[Thm. 24.4]{Conner_Floyd_Differentiable}}.
This result allowed Milnor to show the existence of a surjective ring homomorphism $\mu: \Omega^U_{2*} \to \Omega^O_*$ 
from the unitary bordism ring to the unoriented bordism ring (reducing the dimension by a half), by taking the real points of the set of generators of the unitary bordism ring given by the Milnor hypersurfaces and the complex projective spaces \cite{Milnor-SW}. 
 Milnor deduced from this surjectivity  that an unoriented bordism class contains a complex manifold if and only if it contains a square $N \times N$ \cite[Thm. 1]{Milnor-SW}. Since Milnor was the first to study the properties of the homomorphism $\mu: \Omega^U_{2*} \to \Omega^O_*$, we decided to call this homomorphism the {\it Milnor map}.
 
Notably, the Milnor hypersurfaces that appear in the kernel of the Milnor map have the special property of possessing a free conjugation; namely, an involution that anticommutes with the almost complex structure of the complex manifold, and such that the action of the involution is fixed point free. Motivated by this, Stong organized this extra structure in the framework of bordism, defined the bordism group of unitary manifolds with free conjugations, and showed that the kernel of the Milnor map was precisely the image of the forgetful homomorphism of this bordism group into the unitary bordism group \cite[Lem. 3]{Stong_manifolds_with_conjugation}.

It is worth noting that the above results of Milnor and Stong, relied heavily on knowing an explicit set of generators for both the unitary bordism group and the unoriented bordism group. The Milnor hypersurfaces, along with  their real counterparts, constitute the set of generators having enough extra symmetries which allow us to completely understand the behavior of the Milnor map.

In this work we put forward the idea of generalizing the construction of the Milnor map to the equivariant setup and we present preliminary results in this direction.
We consider the torus group $T^k$ and its real subgroup $\mathbb{Z}_2^k$, along with the equivariant bordism groups $\Omega^{U,T^k}_{2*}$, $\Omega^{U,\mathbb{Z}_2^k}_{2*}$ and 
$\Omega^{O,\mathbb{Z}_2^k}_{*}$,
which are the unitary equivariant bordism groups of $T^k$
 and $\mathbb{Z}_2^k$, and the unoriented equivariant bordism group
of $\mathbb{Z}_2^k$, respectively.
We formulate our idea as a conjecture:
\begin{conjecture*}[Equivariant Milnor map conjecture]
Whenever $G$ is either $T$ or $\mathbb{Z}_2$,  there exists an equivariant Milnor map
\begin{equation}    
\mu^{G^k}: \Omega^{U,G^k}_{2*} \to 
\Omega^{O,\mathbb{Z}_2^k}_{*}
\end{equation}
reducing the dimension by a half, 
defined as the real points of smooth complex generators,
which is moreover surjective.
\end{conjecture*}

We present two results that support the veracity of the conjecture. 
First, we show that the conjecture holds for the Milnor maps $\mu^T$ in Thm. \ref{Milnor_map_semi-free} and  $\mu^{\mathbb{Z}_2}$ in Cor. \ref{Milnor_map_Z2} by constructing explicit complex projective generators for both the unitary bordism group of semi-free circle actions $\overline{\Omega}^{U,T}_{2*}$
and the unitary bordism group of $\mathbb{Z}_2$-actions.
Next, we show that the equivariant Milnor map $ \widehat{\mu}^{T^n} : \widehat{\Omega}_{2n}^{U, T^n} \rightarrow 
  \widehat{\Omega}_{n}^{O, \mathbb{Z}_2^n}$ exists and is surjective in Thm. \ref{Milnor_map_fully_effective} 
  once restricted to the equivariant bordisms of fully effective $T^n$- and $\mathbb{Z}_2^n$-actions, and further restricted to the manifolds of the appropriate dimension. Here we 
  employ  results from  the third author and collaborators in the study of unitary torus manifolds.

If the equivariant Milnor map $\mu^{\mathbb{Z}_2^k}$ exists, we expect the kernel to lie in the image of the $\mathbb{Z}_2^k$-equivariant bordism group of unitary manifolds with free conjugations. This bordism group is defined as the magnetic $\mathbb{Z}_2^k\times \overline{\mathbb{Z}}_2$-equivariant bordism group $\mathbf{\Omega}^{\mathbb{Z}_2^k \times \overline{\mathbb{Z}}_2}_*(E\overline{\mathbb{Z}}_2) $ associated to the classifying space $E\overline{\mathbb{Z}}_2$ for free $\overline{\mathbb{Z}}_2$-actions.
Accordingly, assuming the existence of such equivariant Milnor map $\mu^{\mathbb{Z}_2^k}$,
we expect the  sequence of homomorphisms
\begin{equation}    
\mathbf{\Omega}^{\mathbb{Z}_2^k \times \overline{\mathbb{Z}}_2}_{2*}(E\overline{\mathbb{Z}}_2) \stackrel{}{\longrightarrow} \Omega^{U,\mathbb{Z}_2^k}_{2*} \stackrel{\mu^{\mathbb{Z}_2^k}}{\longrightarrow}
\Omega^{O,\mathbb{Z}_2^k}_{*} \to 0,
\end{equation}
to be exact, where the left-hand morphism is given by forgetting the free involution structure.

We have confirmed the existence of this exact sequence for the cases for $k=0$ in Thm. \ref{theorem_sequences_Milnor_map} and $k=1$ in Thm.  \ref{theorem_sequence_Z2_Milnor_map}
by constructing  explicit equivariant unitary manifolds with free conjugations that generate the kernel of the Milnor map. We emphasize here that prior knowledge of the explicit manifold representatives of the kernel of the equivariant Milnor map is essential.

In order to obtain and present the stated results, we have organized this work into four sections. 

In the first section, we extend to the equivariant setting some fundamental bordisms relating equivariant manifolds with  involution to equivariant manifolds constructed out of the fixed points of an involution. We then present the Milnor hypersurfaces and we rework the proof due to Milnor, of the surjectivity of the Milnor map. 

In the second section, we present a set of generators for the unitary bordism group of semi-free circle actions, and similarly, for the unoriented $\mathbb{Z}_2$-equivariant bordism groups. As previously noted, knowing the explicit sets of generators allowed us to prove the surjectivity of the equivariant Milnor map for the circle and the cyclic group of order two.

In the third section, we provide a summary of results for quasitoric manifolds which, once suitably organized, enabled us to show the surjectivity of the equivariant Milnor map $ \widehat{\mu}^{T^n}$  when restricted to the unitary bordisms of fully effective actions of the torus $T^n$. Here the surjectivity is only proved for bordism groups of $T^n$-equivariant unitary manifolds of dimension $2n$.

In the fourth and last section we introduce the magnetic unitary bordism groups and we highlight  some of its properties. We then restrict the magnetic unitary bordism groups to those of free conjugations and construct a set of generators for this bordism group. We show that the kernel of the Milnor map is precisely the image of the forgetful map of the magnetic unitary bordism of free conjugations, and we construct a concrete set of manifolds realizing this kernel. We generalize this proof to the $\mathbb{Z}_2$-equivariant Milnor map, and similarly, we exhibit a concrete set of manifolds realizing this kernel.

\section{Equivariant bordisms and fixed points}
\subsection{Fundamental cobordisms}

Here we will present a series of bordisms that relate an equivariant manifold with the normal bundle of its fixed points.

We will be interested in two types of manifolds. Unitary manifolds endowed with a compatible action of the torus $T^k$, and unoriented manifolds endowed with an action of the group $\mathbb{Z}_2^k$.
Denote the corresponding bordism rings $\Omega^{U, T^k}_*$ and $\Omega^{O,\mathbb{Z}_2^k}_*$
respectively (see \cite{LuWang}
for the definitions).

The following cobordisms are fundamental:

\begin{lemma} \label{lemma_bordism_Tk}
    Let $M$ be a $T^k$-equivariant unitary manifold whose isotropy groups are all connected. Let $\widehat{T} \hookrightarrow  T^k$ be any injective group homomorphism where $\widehat{T}\cong S^1$, $F =M^{\widehat{T}}$ its fixed point set, and $E \to F$ its normal bundle.
Then there is a $T^k\times \widehat{T}$-equivariant cobordism between $S^1 \times M$ and the sphere bundle $S(E\oplus \mathbb{C})$. Here
$T^k$ acts on $M$, $\widehat{T}$ acts freely on $S^1$, and $\widehat{T}$ acts diagonally on $E \times \mathbb{C}$. Namely, $\widehat{T}$ 
acts on $E$ by the homomorphism $\widehat{T} \to T^k$ and on $\mathbb{C}$ by complex multiplication.
\end{lemma}

\begin{proof}

    Take two copies of the unitary manifolds $B(\mathbb{C})\times M$, where $B(\mathbb{C})$ denotes the ball of radius one in the complex numbers, and endow them
    with the following $\widehat{T}$-actions.
For $t \in  \widehat{T} $  and $(\lambda, m) \in B(\mathbb{C}) \times M$ we define the action on the first copy to be: \begin{align}
    t \cdot ( \lambda, m) = (t^{-1} \lambda, m)
\end{align}
while on the second copy the action is:
\begin{align}
    t \cdot ( \lambda, m) = (t \lambda, tm).
\end{align}
Here we are making an abuse of notation. The element $t \in \widehat{T}$ acts on $\mathbb{C}$ according to a prescribed isomorphism $\iota: \widehat{T} \overset{\cong}{\to} S^1$. The actions should have read $t \cdot (\lambda,m) = (\iota(t^{-1}) \lambda, m)$ and $t \cdot (\lambda,m) = (\iota(t) \lambda, tm)$. To keep the notation simple, 
we will disregard the isomorphism $\iota$; it will be clear from the context.

On the boundaries of the two manifolds we define the isomorphism
\begin{align}
 \phi:   S(\mathbb{C}) \times M  &\to  S(\mathbb{C}) \times M \\
    (\lambda,m) & \mapsto (\lambda^{-1}, \lambda^{-1}m)
\end{align}
where $\lambda^{-1}m$ denotes the element $\iota^{-1}(\lambda^{-1})m$,
and $S(\mathbb{C})$ denotes the complex numbers of unit norm. This isomorphism $\phi$ is compatible with the $\widehat{T}$-action  since we have the equations:
\begin{align}
    \phi(t \cdot (\lambda,m))&= \phi( t^{-1}\lambda, m) =(t \lambda^{-1}, t \lambda^{-1}m)  \\
    t \cdot \phi(\lambda,m)& = t \cdot (\lambda^{-1}, \lambda^{-1}m) = (t \lambda^{-1}, t\lambda^{-1}m), 
\end{align}
and it is moreover compatible with the $T^k$-action which only affects the coordinates of $M$ and commutes with the $\widehat{T}$-action.

Glue the two manifolds along their boundary using the isomorphism $\phi$ to obtain the manifold:
\begin{align}
    Z =( B(\mathbb{C}) \times M) \ \underset{\phi}{\sqcup}\  ( B(\mathbb{C}) \times M)
\end{align}
which is $T^k \times \widehat{T}$-equivariant with unitary action. The $\widehat{T}$-fixed points are:
\begin{align}
    Z^{\widehat{T}}=M \sqcup F
\end{align}
with normal bundles isomorphic to
\begin{align}
    \mathbb{C} &\times M \to M & E \oplus \mathbb{C} \to F
\end{align}
where $E$ denotes the normal bundle of $F$ in $M$.

Observe that the fixed point set $F = \sqcup F_m$ might have different connected components $F_m$ of different dimensions. The normal bundle of each connected component $F_m$ may be denoted as $E_m$
and we will adopt the convention that the bundle $E \to F$ reads as the disjoint union of the bundles $E_m \to F_m$.

Removing tubular neighborhoods of the $\widehat{T}$-fixed points on the manifold $Z$, we obtain a $T^k \times \widehat{T}$-equivariant  cobordism from $S(\mathbb{C}) \times M$ to the sphere bundle 
$S(E \oplus \mathbb{C})$.
\end{proof}

In the case that the $\widehat{T}$-action were free on the whole cobordism, we would obtain a $T^k$-equivariant cobordism between $M$ and the quotient manifold:
\begin{align}
    S(E\oplus \mathbb{C})/\widehat{T}.
\end{align}
By unwinding the definition of this quotient space, we are led to the concept of a twisted projective space, which we describe in the following definition.

\begin{definition} \label{definition_twisted_projectivization} Let $E \to F$ be a complex $T^k \times \widehat{T}$-equivariant vector bundle where the fixed point of $\widehat{T} \cong S^1$ is the zero section isomorphic to $F$.
Split $E=E^+ \oplus E^{-} $ where  $\widehat{T}$ acts on $E^+$ with positive winding number and on $E^-$ with negative winding number.
 For $r \geq 0$ consider the action of $\widehat{T}$ on $E^+ \oplus E^- \oplus \mathbb{C}^r$ by the equation:
    \begin{align}
    s \cdot (v^+,v^-,z) = (sv^+,s^{-1}v^{-}, sz)
    \end{align}
    for $s \in \widehat{T}$ and $(v^+,v^-,z) \in E^+ \oplus E^- \oplus \mathbb{C}^r$. The projectivization of this bundle by the $\widehat{T}$-action will be denoted 
    \begin{align}
    \mathbb{P}^t_\mathbb{C}(E^+ \oplus E^- \oplus \mathbb{C}^r) \coloneqq S(E^+ \oplus E^- \oplus \mathbb{C}^r)/ \widehat{T}
    \end{align}
    and will be called {\it twisted projective space}. Note that it inherits the
    action of the group $T^k$ thus defining a $T^k$-equivariant space. 
\end{definition}

Unfortunately, this projective space $\mathbb{P}^t_{\mathbb{C}}(E^+ \oplus E^- \oplus \mathbb{C}^r)$ does not inherit the structure of a manifold unless the group $\widehat{T}$ acts freely on the sphere bundle  $S(E^+ \oplus E^-)$. This happens precisely when the action of $\widehat{T}$ is {\it semi-free}.

\begin{definition}
\label{def_semi_free_actions}
    A circle $S^1$ action on a manifold $M$ is {\it semi-free} if the only isotropy groups are either the full group itself or the trivial group.
\end{definition}

\begin{lemma} \label{lemma_hatT-semifree}
    Assume the same hypothesis of Lem. \ref{lemma_bordism_Tk} and require that the action of $\widehat{T}$ on $M$ is semi-free. There is a $T^k$-equivariant bordism from $M$ to the twisted projective space $\mathbb{P}_\mathbb{C}^t(E^+ \oplus E^- \oplus \mathbb{C})$.
\end{lemma}

\begin{proof}
    The normal bundle $E \to F$ of
a connected component of the fixed points can be split into a Whitney sum of vector bundles
\begin{align}
    E \cong E_{j_1 }\oplus \cdots \oplus E_{j_k}
\end{align}
where $\widehat{T}$ acts on $E_{j_l}$ with winding number $j_l$, that is $t \cdot e = t^{j_l}e$ for $e \in E_{j_l}$ and $t \in \widehat{T}$. Note that $j_l \neq 0$ and are all different.

The action of $\widehat{T}$ on $S(E)$
is free if and only if the only 
winding numbers that can appear are $\pm 1$. This is the local description of the semi-free condition. Therefore, we have 
$E \cong E^+ \oplus E^-$ where $E^+=E_1$ and $E^-=E_{-1}$ and the 
twisted projective space is defined by the action of $\widehat{T}$ on the sphere bundle $S(E^+ \oplus E^{-} \oplus \mathbb{C})$:
\begin{align}
    s\cdot (\nu^+,\nu^-,z) = (s\nu^+, s^{-1}\nu^-,sz).
\end{align}
Since the action is free, the
quotient
$\mathbb{P}^t_\mathbb{C}(E^+ \oplus E^- \oplus \mathbb{C})$
is a manifold. The action of $ t \in T^k$ in homogeneous coordinates is given by the following equation:
\begin{align}
    t \cdot [\nu^+:\nu^-:z] = [t\nu^+:t\nu^-:z].
\end{align}
The $T^k$-equivariant bordism between $M$ and $\mathbb{P}^t_\mathbb{C}(E^+ \oplus E^- \oplus \mathbb{C})$ follows from Lem. \ref{lemma_bordism_Tk} since
the cobordism described in the proof of the lemma between $M \times S^1$ and the sphere bundle $S(E^+ \oplus E^{-} \oplus \mathbb{C})$ has a free $\widehat{T}$-action.
\end{proof}

The unoriented version of Lem. \ref{lemma_bordism_Tk} is the following:
\begin{lemma} \label{lema_bordism_Z2k}
    Let $N$ be a manifold endowed with an action of $\mathbb{Z}_2^k$. Let $\tau \in  \mathbb{Z}_2^k$ be any non-trivial element and consider $E \to F$ the normal bundle of the fixed points $F=N^\tau$, then $N$ and $\mathbb{P}_\mathbb{R}(E \oplus \mathbb{R})$ are $\mathbb{Z}_2^k$-equivariantly cobordant.
\end{lemma}
\begin{proof}
    Consider two copies of $[-1,1] \times N$. Endow them with an extra  $\mathbb{Z}_2$-structure, on the first copy we have 
    $(t,n) \mapsto (-t,n)$, while of the second we have $(t,n) \mapsto (-t,\tau (n))$.
Glue their boundaries with the map $\phi : \{-1,1\} \times N \to \{-1,1\} \times N$, defined by the equations $\phi(1,n)=(1,n)$, $\phi(-1,n)=(-1,\tau(n))$.

The manifold $W= ([-1,1] \times N)\  \underset{\phi}{\sqcup} \ ([-1,1] \times N)$ is endowed with an action of the group $\mathbb{Z}_2^k \times \mathbb{Z}_2$, the group $\mathbb{Z}_2^k$ acts on the coordinates of $N$, while the copy of $\mathbb{Z}_2$ on the right acts on $W$ as described above. 

The fixed points of this $\mathbb{Z}_2$-action on $W$ returns:
\begin{align}
    W^{\mathbb{Z}_2} = N \sqcup F 
\end{align}
    and their normal bundles are isomorphic to $N \times\mathbb{R} \to N$ and $E \times \mathbb{R}  \to F$ respectively. Removing tubular neighborhoods around the fixed point sets, and noting that the action of $\mathbb{Z}_2$ is free 
on this space, the quotient by $\mathbb{Z}_2$ is a cobordism between
$N$ and $S(E \oplus \mathbb{R})/\mathbb{Z}_2= \mathbb{P}_\mathbb{R}(E \oplus \mathbb{R})$. Since the cobordism
has a compatible action of the group $\mathbb{Z}_2^k$, the result follows.
\end{proof}

Now, we want to apply the previous lemma to the case on which we have an almost complex manifold $(M,J)$
endowed with an involution $\tau : M \to M$ which behaves like complex conjugation. Recall that an almost complex manifold $M$ is a smooth manifold endowed with and endomorphism $J :TM \to TM$ such that $J^2=-1$.

\begin{definition}
    An involution $\tau : M \to M$ on an almost complex manifold $(M,J)$ is called a {\it conjugation} if $\tau_* J = - J \tau_*$.
\end{definition}

\begin{proposition} \label{generalization_M=FxF}
    Let $(M,J)$ be an almost complex structure on the closed manifold $M$ and $\tau: M \to M$ a conjugation for this complex structure. Assume further, that the group $\mathbb{Z}_2^k$ acts on $M$ commuting with $\tau$ and compatible with $J$ (unitarily on $TM$), thus having an action of the group $\mathbb{Z}_2^k \times \mathbb{Z}_2$. Let $F=M^\tau$ be the fixed point set of the involution $\tau$ and endow the product $F \times F$ with the 
    the diagonal action of $\mathbb{Z}_2^k$ and the action of $\mathbb{Z}_2$ that flips the coordinates $(x,y) \mapsto (y,x)$. Then the manifolds $M$ and $F\times F$ are $\mathbb{Z}_2^k \times \mathbb{Z}_2$-equivariantly cobordant.
\end{proposition}

\begin{proof}
The compatibility of the $\mathbb{Z}_2^k$-action
means that for $g \in \mathbb{Z}_2^k$ we have $g \tau=\tau g$ and $g_* J = J g_*$. Therefore, the fixed point set $M^\tau=F=\sqcup_mF_m$ is endowed with an action of $\mathbb{Z}_2^k$, and its normal bundle $\sqcup_m E_m \to \sqcup F_m$ is a 
$\mathbb{Z}_2^k$-equivariant bundle. Note that
the normal bundle $E_m$ may be taken to be $J(TF_m) \subset TM$, and therefore the normal bundle $\sqcup_m E_m \to \sqcup F_m$ is 
$\mathbb{Z}_2^k$-equivariantly isomorphic to the bundle $TF \to F$.
The extra action of $\mathbb{Z}_2$ acts trivially on the base $F$ and by the sign representation on the fibers of $TF$.

We apply Lem. \ref{lema_bordism_Z2k} to the $\mathbb{Z}_2^k \times \mathbb{Z}_2$-action on $M$ where the involution $\tau$ generates the
second group $\mathbb{Z}_2$ and we obtain a $\mathbb{Z}_2^k \times \mathbb{Z}_2$-equivariant cobordism between $M$ and $\mathbb{P}_\mathbb{R}(TF \oplus \mathbb{R})$.

On the other hand, the diagonal action of $\mathbb{Z}_2^k$ on $F\times F$ makes it a $\mathbb{Z}_2^k$-space,
commuting with involution $(x,y)\mapsto (y,x)$.
The fixed point set of the involution is the diagonal $\Delta=\{(x,x) | x \in F \}$, which is $\mathbb{Z}_2^k$-equivariantly isomorphic to $F$. The normal bundle of the diagonal is also $\mathbb{Z}_2^k$-equivariantly isomorphic to $TF \to F$. 
The extra $\mathbb{Z}_2$-action acts trivially on the base, which is the diagonal, and by the sign representation on the fiber of $TF.$

We again apply Lem. \ref{lema_bordism_Z2k} to the $\mathbb{Z}_2^k \times \mathbb{Z}_2$-action on $F \times F$
and we obtain an $\mathbb{Z}_2^k \times \mathbb{Z}_2$-equivariant cobordism between $F \times F$ and $\mathbb{P}_\mathbb{R}(TF \oplus \mathbb{R})$.

Since for both  $M$ and $F \times F$, the normal bundles of
the fixed points of the involution
$\tau$ are $\mathbb{Z}_2^k\times \mathbb{Z}_2$-equivariantly isomorphic, we can glue the cobordisms along the projectivization
thus obtaining the desired $\mathbb{Z}_2^k \times \mathbb{Z}_2$-equivariant cobordism between $M$ and $F \times F$.
\end{proof}

\subsection{Milnor map}
The previous result of Prop. \ref{generalization_M=FxF} generalizes Conner and Floyd's result {\cite[Thm. 24.4]{Conner_Floyd_Differentiable}}
to the equivariant setup. Applied to the context of smooth complex algebraic varieties $V_\mathbb{C}$ endowed with a unitary action of $T^k$, we can say that the real points $V_\mathbb{R}$ are endowed with an action of $\mathbb{Z}_2^k \subset T^k$, and, if the real points $V_\mathbb{R}$ are smooth, the manifold $V_\mathbb{C}$ is cobordant to $V_\mathbb{R} \times V_\mathbb{R}$
in the $\mathbb{Z}_2^k \times \mathbb{Z}_2$-equivariant unoriented bordism ring. By forgetting the conjugation involution, we have that
$[V_\mathbb{C}]_2=[V_\mathbb{R} \times V_\mathbb{R}]_2$ in $\Omega^{O, \mathbb{Z}_2^k}_*$.

In the non-equivariant setting, a smooth complex algebraic variety $V_\mathbb{C}$ whose real points $V_\mathbb{R}$ are smooth, is cobordant to the product $V_\mathbb{R} \times V_\mathbb{R}$.
Motivated by this remarkable result, Milnor noted that the image of the forgetful homomorphism
\begin{align} 
  F:  \Omega^{U}_* \to  \Omega^{O}_*, \ [M] \mapsto [M]_2
\end{align}
from the unitary bordism ring $\Omega^{U}_*$ to the unoriented bordism ring $\Omega^{O}_*$, is precisely the subring
generated by all squares $[N \times N]_2 \in \Omega^{O}_{2*}$ \cite[Thm. 1]{Milnor-SW}.

Since a (redundant) set of generators of the unitary bordism ring may be taken to be the complex manifolds $\mathbb{C}P^n$ and the hypersurfaces:
\begin{align} \label{Milnor_hypersurfaces_definition}
    H_\mathbb{C}(m,n) : = \{\big([\mathbf{w}], [\mathbf{z}]\big) \in \mathbb{P}_\mathbb{C}^m \times \mathbb{P}_\mathbb{C}^n \ | \ w_0z_0 + \cdots +w_mz_m=0 \}
\end{align}
for $m \leq n$, and their real points $\mathbb{R}P^n$ and $H_\mathbb{R}(m,n)$ respectively (fixed point sets of the complex conjugation involutions) are also a (redundant) set of generators for the unoriented bordism ring, Milnor was able to show that there is a surjective ring homomorphism:
\begin{align}
\mu: \Omega_{2*}^U \to \Omega_{*}^O
\end{align}
defined by the real points of the complex manifolds that generate the unitary bordism ring:
\begin{align}
[\mathbb{P}_\mathbb{C}^m] \mapsto [\mathbb{P}_\mathbb{R}^m]_2, \ &\ [H_\mathbb{C}(m,n)] \mapsto [H_{\mathbb{R}}(m,n)]_2.
\end{align}

\begin{definition}
Call the surjective ring homomorphism $\mu:\Omega_{2*}^U \to \Omega_{*}^O$
the {\bf Milnor map}.
\end{definition}
Milnor further showed an explicit set of generators that realizes this map. Let us recall this proof.

\begin{proposition}[{\cite[\S 1]{Milnor-SW}}] The Milnor map modulo 2:
\begin{align}
 \mu{\otimes}\mathbb{Z}_2:   \Omega^{U}_{2*} {\otimes}\mathbb{Z}_2 \to \Omega^{O}_* 
\end{align}
is a split homomorphism of rings. This follows from the following results:
\begin{itemize}
\item The unoriented bordism ring $\Omega^{O}_* $ is a polynomial ring over $\mathbb{Z}_2$ with independent generators $\mathbb{P}_\mathbb{R}^{2r}$ and $H_\mathbb{R}(2^k,2t2^k)$ for $r,k,t \geq 1$.
\item The modulo 2 unitary bordism ring $\Omega^{U}_{2*} {\otimes}\mathbb{Z}_2$ is a polynomial ring over $\mathbb{Z}_2$ with independent generators $\mathbb{P}_\mathbb{C}^{2r}$, $H_\mathbb{C}(2^k,2t2^k)$ and $H_\mathbb{C}(2^{s-1},2^{s-1})$ for $r,k,t,s \geq 1$.
\end{itemize}    
\end{proposition}
\begin{proof}
    Recall that the characteristic numbers of the Milnor hypersurfaces are given by:
    \begin{equation}
    s_{m+n-1}\left(H_\mathbb{C}(m,n) \right) = 
        \begin{cases}
            2 & m = n = 1 \\
            0 & m = 1, n > 1 \\
            - \binom{m+n}{m} & m > 1.
         \end{cases}
     \end{equation}
This number modulo 2 is the same for the real Milnor hypersurfaces $H_\mathbb{R}(m,n)$. Now, Thom \cite{Thom} showed that the even-dimensional generators of the unoriented bordism ring are the projective spaces $\mathbb{P}_\mathbb{R}^{2r}$, while for the odd-dimensional generators it is enough to consider manifolds of dimension different from $2^j -1$ whose characteristic numbers are odd. Any odd number that is not of the form $2^j -1$
is of the form $2^k(2t+1)-1$ for some $k,t \geq 1$. The Milnor hypersurface
$H_{\mathbb{R}}(2^k,2t2^k)$ has dimension $2^k(2t+1)-1$ and a simple calculation shows that its characteristic number is odd. Therefore, the Milnor hypersurfaces
$H_{\mathbb{R}}(2^k,2t2^k)$ together with the real projective spaces $\mathbb{P}_\mathbb{R}^{2r}$ form an independent set of generators for the unoriented bordism ring.

For the unitary case, we can take the even-dimensional complex projective spaces $\mathbb{P}_\mathbb{C}^{2r}$   
as generators since their characteristic numbers are odd. For  manifolds with odd complex dimension we can take  the Milnor hypersurfaces
$H_{\mathbb{C}}(2^k,2t2^k)$ together with the hypersurfaces $H_{\mathbb{C}}(2^{s-1},2^{s-1})$. The former ones have odd characteristic number as shown above, while the characteristic number of the latter ones is twice an odd number. Recall that
a unitary manifold  of dimension $2^{s}-1$ is an algebra generator of the unitary bordism ring if its characteristic number is $2$ \cite{Novikov}. Therefore $H_{\mathbb{C}}(2^{s-1},2^{s-1})$ is a generator of complex dimension $2^{s}-1$ of the modulo 2 unitary bordism ring.

We therefore have that the kernel of the Milnor map modulo 2 is the ideal generated by the manifolds $H_{\mathbb{C}}(2^{s-1},2^{s-1})$. Moreover, since the ideal $2\Omega^{U}_{2*}$ is mapped to zero by the Milnor map we get that the kernel of the Milnor map is the ideal generated by the manifolds $H_{\mathbb{C}}(2^s,2^s)$ and the ideal $2\Omega^{U}_{2*}$.

\begin{lemma} \label{lemma_kernel_mu}The kernel of the Milnor map is the following ideal:
\begin{align} \label{kernel mu}
   \mathrm{ker}(\mu : \Omega^{U}_{2*} \to \Omega^{O}_*)  = 2\Omega^{U}_{2*} +\Omega^{U}_{2*}\langle [H_{\mathbb{C}}(2^{s-1},2^{s-1})] \colon s \geq 1 \rangle 
\end{align}
\end{lemma}

The assignment 
\begin{align} 
\Omega^{O}_* & \to \Omega^{U}_{2*} \otimes \mathbb{Z}_2\\
H_{\mathbb{R}}(2^k,2t2^k) & \mapsto H_{\mathbb{C}}(2^k,2t2^k) \\
\mathbb{P}_\mathbb{R}^{2r} & \mapsto
\mathbb{P}_\mathbb{C}^{2r} 
\end{align}
induces a ring homomorphism since both rings are polynomial rings in independent generators.  The proposition follows.
\end{proof}

Note that the forgetful homomorphism becomes the square of the Milnor map $\mu$, thus producing the relation:
\begin{align}
[V_\mathbb{C}]_2=[V_\mathbb{R} \times V_\mathbb{R}]_2 \in \Omega^{O}_{2*}.
\end{align}

An equivariant version of this homomorphism is one of the aims of this work. To define this map, we need to generalize to the equivariant setup certain properties used in the definition of the Milnor map. In particular we have:
\begin{itemize}
    \item[(a)] Both bordism rings $\Omega^U_*$ and $\Omega^O_*$ are free polynomial algebras. Namely, $\mathbb{Z}[y_{2j} | j \in \mathbb{N}^{>0}]$ and $\mathbb{Z}_2[x_k | k \in \mathbb{N}^{>0},  k \neq 2^s-1]$ respectively.
    \item[(b)] The generators of the unitary bordism ring may be taken to be complex manifolds. Then the Milnor map is simply defined as the real points of those manifolds.
    \item[(c)] The real points of the complex manifolds that generate the unitary bordism group also generate the unoriented bordism ring.
   
\end{itemize}
These three properties have to be accommodated to the equivariant setup. 
\begin{remark}
In the quasitoric context, Davis and Januszkiewicz  \cite{Davis-Januszkiewicz}  show that each quasitoric manifold $M^{2n}$ over a simple convex polytope $P^n$ always admits a natural conjugation involution $\tau$ whose fixed point set $M^\tau$ is just a small cover over $P^n$. In particular,  this conjugation involution   $\tau$ is independent of the choices of omniorientations on $M^{2n}$, and by  \cite[Cor. 6.7 \& 6.8]{Davis-Januszkiewicz}, one has that the mod 2 reductions of all Chern numbers of $M^{2n}$ with an omniorientation determine all Stiefel--Whitney numbers of $M^\tau$, and in particular,  $[M^{2n}]_2=[M^\tau\times M^\tau]_2$ as unoriented bordism classes in $\Omega_*^O$.
Thus, $\tau$ induces a homomorphism $ \Omega^U_{2n}\to \Omega_n^O$, which exactly agrees with the above homomorphism $\mu: \Omega^U_{2n}\to \Omega_n^O$.
\end{remark}

\section{Semi-free actions of the circle and unitary involutions}

Consider the equivariant geometric bordism groups $\Omega^{U,T}_*$ and $\Omega_*^{O,\mathbb{Z}_2}$ of equivariant unitary manifolds with circle actions  and equivariant unoriented manifolds with involutions \cite{LuWang}. One can define a homomorphism $\Omega^{U,T}_{2*} \to \Omega_*^{O,\mathbb{Z}_2}$,
using the explicit set of generators for the module $\Omega^{U,T}_{2*}$ presented in    \cite{Hattori_circle, Kosniowski-circle}
that satisfy the properties listed above. 

Instead, we will show that 
it is enough to consider
 a small class of $T$-equivariant manifolds in order to show the surjectivity to  $\Omega_*^{O,\mathbb{Z}_2}$.
 The manifolds of interest are
$T$-equivariant unitary manifolds where the action is semi-free (see Def. \ref{def_semi_free_actions}).

\subsection{Unitary bordism of semi-free circle actions}

\begin{definition}
    Let $\overline{\Omega}^{U,T}_*$ be the unitary bordism ring of manifolds with semi-free circle actions.
\end{definition}
This bordism ring has been extensively studied and it has served as one first candidate to completely determine the ring structure of the equivariant bordism ring \cite{Sinha-semi-free}.

We claim that the generalization of the Milnor map restricted to semi-free actions
\begin{align}
    \mu : \overline{\Omega}^{U,T}_{2*} \to {\Omega}^{O,\mathbb{Z}_2}_*
\end{align}
is surjective. Let us elaborate on how this map is constructed.

First note that both bordism rings $\overline{\Omega}^{U,T}_{2*} $ and ${\Omega}^{O,\mathbb{Z}_2}_*$ are free modules over the rings $\Omega^U_*$ and $\Omega^O_*$ respectively \cite{Loeffler_unitary, Buchstaber-Panov-Ray, LuWang}. This takes care of condition (a) since we will define the equivariant Milnor map using the generators as modules  
rather than rings.

To determine a set of generators for both bordism groups $\overline{\Omega}^{U,T}_{2*}$ and ${\Omega}^{O,\mathbb{Z}_2}_*$ we will make use of Conner--Floyd's long exact sequence for pairs of families of subgroups \cite{Stong_complex_and_oriented}. For both $T$ and $\mathbb{Z}_2$ we let $\mathcal{A}$ denote the family of all subgroups and $\{\{1\}\}$ the family with only the trivial subgroup, and take the associated long exact sequences:
\begin{align} \label{short-exact-sequence-relative-semi-free}
  0  \to \overline{\Omega}^{U,T}_{2*} \to \overline{\Omega}^{U,T}_{2*}\{\mathcal{A}, \{1\}\} \overset{\partial}{\to}  \overline{\Omega}^{U,T}_{2*-1}\{\{1\}\} \to 0\\
  0  \to {\Omega}^{O,\mathbb{Z}_2}_{*} \to {\Omega}^{O,\mathbb{Z}_2}_{*}\{\mathcal{A}, \{1\}\} \overset{\partial}{\to}  {\Omega}^{O,\mathbb{Z}_2}_{*-1}\{\{1\}\} \to 0
\end{align}
which in both cases split as short exact sequences. This is because we can find sections for the boundary maps by the following argument: A manifold $[E]$ in  $\overline{\Omega}^{U,T}_{2*-1}\{\{1\}\}$
has a free action of $T$ and therefore it is a principal $T$-bundle. Take the associated
complex line bundle $E \times_T \mathbb{C}$ and note that the ball bundle $B(E \times_T \mathbb{C})$ of vectors of norm less than or equal to $1$ defines a bordism class in 
$\overline{\Omega}^{U,T}_{2*}\{\mathcal{A}, \{1\}\}$. Thus, we have the sections:
\begin{align}
   \overline{\Omega}^{U,T}_{2*-1}\{ \{1\}\} & \to \overline{\Omega}^{U,T}_{2*}\{\mathcal{A}, \{1\}\} & {\Omega}^{O,\mathbb{Z}_2}_{*-1}\{ \{1\}\} & \to {\Omega}^{O,\mathbb{Z}_2}_{*}\{\mathcal{A}, \{1\}\} \\
   [E]& \mapsto [B(E \times_T \mathbb{C)]} & [H] & \mapsto [B(H \times_{\mathbb{Z}_2} \mathbb{R})]
\end{align}
where in both cases $\partial B(E \times_T \mathbb{C)} \cong E$ and $\partial  B(H \times_{\mathbb{Z}_2} \mathbb{R}) \cong H$.

Now we will construct projection maps 
\begin{align}
    \overline{\Omega}^{U,T}_{2*}\{\mathcal{A}, \{1\}\} &\to \overline{\Omega}^{U,T}_{2*} &
      {\Omega}^{O,\mathbb{Z}_2}_{*}\{\mathcal{A}, \{1\}\} &\to {\Omega}^{O,\mathbb{Z}_2}_{*}
\end{align}
using the fact that the relative bordism groups of the pair of families $\{\mathcal{A}, \{1\}\}$ can be written in terms of non-equivariant bordism groups.

Take a manifold with boundary $(N, \partial N)$ defining a class in $\overline{\Omega}^{U,T}_{2*}\{\mathcal{A}, \{1\}\}$; this is a unitary manifold $N$ with a compatible action of $T$ such that $T$ acts freely on $\partial N$. Let $F=N^T$ be the fixed point set of the action, and write it as a union of connected components $F = \cup_m F_m$ where $n_m = \mathrm{dim}_\mathbb{R}(F_m)$. Clearly, each connected component $F_m$ is a closed submanifold embedded in the interior of $N$ since $\partial F=F\cap \partial N$ is empty. Denote $E_m \to F_m$ the normal bundles of $F_m$ in $N$ and note that the circle action splits each $E_m$ into the Whitney sum  $E_m^+\oplus E_m^-$ where $T$ acts on $E_m^+$ by multiplication and on $E_m^-$ by multiplication of the inverse. Denote the rank of these vector bundles by $n_m^\pm\coloneqq \mathrm{rank}_\mathbb{C}E_m^\pm$.

There is a cobordism
\begin{align}
    (N,\partial N) \sim \bigcup_m \big(B(E_m^+ \oplus E_m^-),S(E_m^+ \oplus E_m^-)\big)
\end{align}
which allows us to make the identification
\begin{align}
    \overline{\Omega}^{U,T}_{2n}\{\mathcal{A}, \{1\}\} \cong &\bigoplus_{k+k^++k^-=n} \Omega^{U}_{2k}(BU(k^+)\times BU(k^-)) \\
    [(N,\partial N)] \mapsto & \sum_{m} [F_m \to BU(n_m^+)\times BU(n_m^-)]
\end{align}
where the maps $F_m \to BU(n_m^+)\times BU(n_m^-)$ classify the vector bundles $E_m^+ \oplus E_m^-$.

We define the homomorphism
\begin{align}
    \Omega^{U}_{2k}(BU(k^+)\times BU(k^-)) &\to \overline{\Omega}^{U,T}_{2n} \\
    [F \to BU(k^+) \times BU(k^-)]  &\mapsto [\mathbb{P}^t_\mathbb{C}(E^+ \oplus E^- \oplus \mathbb{C})]
\end{align}
where $E^+ \oplus E^-$ is the vector bundle that the map $F \to BU(k^*) \times BU(k^-)$ classifies, and $\mathbb{P}^t_\mathbb{C}(E^+ \oplus E^- \oplus \mathbb{C})$ is the twisted projective space of Def. \ref{definition_twisted_projectivization} and appearing in Lem. \ref{lemma_hatT-semifree}.

The twisted projective space $\mathbb{P}^t_\mathbb{C}(E^+ \oplus E^- \oplus \mathbb{C})$  inherits a $T$-action which is given by the following equation in homogeneous coordinates:
\begin{align}
    t \cdot [v^+:v^-:z] = [tv^+:t^{-1}v^-:z].
\end{align}

The $T$-action on the twisted projective space is semi-free and its $T$-fixed points are:
\begin{align}
    \mathbb{P}^t_\mathbb{C}(E^+ \oplus E^- \oplus \mathbb{C})^T = F \sqcup \mathbb{P}^t_\mathbb{C}(E^+ \oplus E^- ),
\end{align}
with normal bundles isomorphic to the following bundles:
\begin{align}
    E^+ \oplus E^- &\to F & \mathbb{C} \times \mathbb{P}^t_\mathbb{C}(E^+ \oplus E^- ) & \to \mathbb{P}^t_\mathbb{C}(E^+ \oplus E^- ).
\end{align}
The $T$-action on the first bundle matches the original $T$-action on $E^+ \oplus E^-$, while on the second the $T$-action is given by the following equation:
\begin{align}
    t \cdot (z, [v^+:v^-]) = (t^{-1}z,[v^+:v^-]),
\end{align}
where $t \in T$, $z \in \mathbb{C}$ and $[v^+:v^-] \in \mathbb{P}^t_\mathbb{C}(E^+ \oplus E^- )$. We obtain the following result which is essentially in \cite{Hattori_circle}.

\begin{lemma} \label{lemma_projective_generatosrs_semi-free}
    Let $M$ be a unitary manifold admitting a semi-free $T$-action. Let $F = \cup_m F_m$ be the fixed point set of the action and let $E_m^+ \oplus E_m^-$ be their normal bundles as defined above. Then the following holds in $\overline{\Omega}^{U,T}_{2n}$
    \begin{align}
        [M] = \sum_m [\mathbb{P}^t_\mathbb{C}(E_m^+ \oplus E_m^- \oplus \mathbb{C})].
    \end{align}
    
\end{lemma}

\begin{proof}
This follows from Lem. \ref{lemma_hatT-semifree} and gives the desired splitting of the relative bordism groups
\begin{align}
\overline{\Omega}^{U,T}_{2*}\{\mathcal{A}, \{1\}\} \ {\cong}  \ \overline{\Omega}^{U,T}_{2*} \oplus   {\Omega}^{U}_{2*-2}(BU(1))
\end{align}
of the short exact sequence
of Eqn. \ref{short-exact-sequence-relative-semi-free}.
\end{proof}

Lemma \ref{lemma_projective_generatosrs_semi-free} allows us to generate elements of the semi-free bordism group with the twisted projective spaces. Therefore, we now need to find a good set of generators for the non-equivariant bordism groups 
 $\Omega^{U}_{2k}(BU(k^+)\times BU(k^-))$. Fortunately this has already been addressed by Conner and Floyd and presented in \cite[pp. 148]{Kochman_birdism}.

The canonical map $BU(1)^n \to BU(n)$ induces a map in bordism $\Omega_*^U(BU(1)^n) \to \Omega_*^U(BU(n))$.
From this map the bordism group $\Omega_*^U(BU(n))$ is a free $\Omega^{U}_*$-module with basis:
\begin{align}
    \{ \alpha_{s_1s_2 \dots s_n} | 0\leq s_1 \leq s_2 \leq \cdots \leq s_n \}
\end{align}
where $ \alpha_{s_1s_2 \dots s_n}$ represents the classifying map of the product of the tautological line bundles
\begin{align} \label{definition_bundle_over_CP}
    \xymatrix{
\gamma_{s_1} \times \cdots \times \gamma_{s_n} \ar[d]& \\
\mathbb{P}_\mathbb{C}^{s_1} \times \cdots \times \mathbb{P}_\mathbb{C}^{s_n} \ar[r] & BU(n).
    }
\end{align}

Therefore $\Omega^{U}_{2k}(BU(k^+)\times BU(k^-))$ is a free $\Omega^{U}_*$-module with basis the classifying maps of the bundles:
\begin{align}
    \xymatrix{
\gamma_{s_1}^+ \times \cdots \times \gamma_{s_{k^+}}^+ \times \gamma_{r_1}^- \times \cdots \times \gamma_{r_{k^-}}^-  \ar[d]& \\
\mathbb{P}_\mathbb{C}^{s_1} \times \cdots \times \mathbb{P}_\mathbb{C}^{s_{k^+}} \times \mathbb{P}_\mathbb{C}^{r_1} \times \cdots \times \mathbb{P}_\mathbb{C}^{r_{k^-}}\ar[r] & BU(k^+) \times BU(k^-).
    }
\end{align}

Therefore we obtain the following result:

\begin{proposition} \label{basis_semi-free_actions}
    Any element in $\overline{\Omega}^{U,T}_{2*}$ can be written as a finite sum of products 
    \begin{align}
        [M_0] \times [\mathbb{P}^t_\mathbb{C}(\gamma_{s_1}^+ \times \cdots \times \gamma_{s_{k^+}}^+ \times \gamma_{r_1}^- \times \cdots \times \gamma_{r_{k^-}}^- \oplus \mathbb{C})]
    \end{align}
    where $[M_0] \in \Omega^{U}_{2*}$ is a unitary manifold with trivial $T$-action, and the twisted projective spaces are defined as above.
\end{proposition}
    
Note that the twisted projective spaces of Prop. \ref{basis_semi-free_actions} define projective fibrations over projective spaces:
\begin{align}
\xymatrix{
    \mathbb{P}_\mathbb{C}(\mathbb{C}^{k^++1} \times \overline{\mathbb{C}}^{k^-}) \ar[r] 
    &\mathbb{P}^t_\mathbb{C}(\gamma_{s_1}^+ \times \cdots \times \gamma_{s_{k^+}}^+ \times \gamma_{r_1}^- \times \cdots \times \gamma_{r_{k^-}}^- \oplus \mathbb{C})
    \ar[d]\\ &
    \mathbb{P}_\mathbb{C}^{s_1} \times \cdots \times \mathbb{P}_\mathbb{C}^{s_{k^+}} \times \mathbb{P}_\mathbb{C}^{r_1} \times \cdots \times \mathbb{P}_\mathbb{C}^{r_{k^-}}.
    }
\end{align}

\subsection{$\mathbb{Z}_2$-equivariant unoriented bordism}

The unoriented version of the previous procedure works without complications. We have the isomorphism:
\begin{align}
    {\Omega}^{O,\mathbb{Z}_2}_{n}\{\mathcal{A}, \{1\}\} \cong &\bigoplus_{k+k'=n} \Omega^{O}_{k}(BO(k')) \\
    [(N,\partial N)] \mapsto & \sum_{m} [F_m \to BO(n_m')]
\end{align}
where $F=\cup_m F_m$ is the fixed point set of the $\mathbb{Z}_2$-action on $N$, $F_m$'s are the connected components, and the map $F_m \to BO(n_m')$ classifies the normal bundle of $F_m$ in $N$ whose rank is $n_m'$.

We can construct closed $\mathbb{Z}_2$-equivariant manifolds from the information on the normal bundles using the following homomorphism:
\begin{align}
    \Omega^{O}_{k}(BO(k')) &\to {\Omega}^{O,\mathbb{Z}/2}_{n} \\
    [F \to BO(k')]  &\mapsto [\mathbb{P}_\mathbb{R}(E \oplus \mathbb{R})],
\end{align}
with $E \to F$ the real bundle classified by the map
$F \to BO(k')$, and $\mathbb{P}_\mathbb{R}(E \oplus \mathbb{R})$ is the projectivization of the bundle $ E\oplus \mathbb{R}$ with $\mathbb{Z}_2$-action defined by the equation:
\begin{align}
    t \cdot [\nu : x] = [tv:x]
\end{align}
where $t =\pm 1$, $\nu \in E $ and $x \in \mathbb{R}$.

The fixed point sets of this projective space are:
\begin{align}
    \mathbb{P}_\mathbb{R}(E \oplus \mathbb{R})^{\mathbb{Z}_2} = F \sqcup \mathbb{P}_\mathbb{R}(E)
\end{align}
with normal bundles isomorphic to the following bundles:
\begin{align}
    E &\to F & \mathbb{R}\times \mathbb{P}_\mathbb{R}(E)  \to \mathbb{P}_\mathbb{R}(E) 
\end{align}
and with the $\mathbb{Z}_2$-actions on both bundles non-trivial.

Analogously as in Lem. \ref{lemma_projective_generatosrs_semi-free}, we have the following result:
\begin{lemma} \label{lemma_projective_generatosrs_unoriented}
    Let $N$ be a closed manifold with $\mathbb{Z}_2$-action. Let $F = \cup_m F_m$ be its fixed point set and $E_m$ the normal bundle of $F_m$ in $N$. Then the following holds in ${\Omega}^{O,\mathbb{Z}_2}_{n}$:
    \begin{align}
        [N] = \sum_m [\mathbb{P}_\mathbb{R}(E_m \oplus \mathbb{R})].
    \end{align}
\end{lemma}

We are now left, once again, with the task of finding a good set of generators for the bordism group $\Omega^{O,\mathbb{Z}_2}_k(BO(k'))$. Thom proved the isomorphism \cite{Thom}:
\begin{align}
    H_*(BO(n); \mathbb{Z}_2) \otimes_{\mathbb{Z}_2} \Omega^{O}_* \cong \Omega^O_*(BO(n))
\end{align}
and the canonical map $BO(1)^{n} \to BO(n)$ induces an isomorphism of groups
\begin{align}
    H_*(BO(1)^{n}; \mathbb{Z}_2)^{\mathfrak{S}_{n}} \cong H_*(BO(n);\mathbb{Z}_2)
\end{align}
where the symmetric group $\mathfrak{S}_{n}$ acts by permutation. This isomorphism follows from the isomorphism in cohomology:
\begin{align}
    H^*(BO(n); \mathbb{Z}_2) \cong \mathbb{Z}_2[\omega_1, \dots, \omega_n] \cong \mathbb{Z}_2[x_1, \cdots, x_n]^{\mathfrak{S}_n} \cong H^*(BO(1)^n; \mathbb{Z}_2)^{\mathfrak{S}_n}.
\end{align}

Since the homology of $BO(1) = \mathbb{P}_\mathbb{R}^\infty$ is generated by the canonical maps $\mathbb{P}_\mathbb{R}^s \to \mathbb{P}_\mathbb{R}^\infty$, we see that
$\Omega_*^O(BO(n))$ is a free $\Omega^{O}_*$-module with basis:
\begin{align}
    \{ \xi_{u_1u_2 \dots u_n} | 0\leq u_1 \leq u_2 \leq \cdots \leq u_n \}
\end{align}
where $ \xi_{u_1u_2 \dots u_n}$ represents the classifying map of the product of the tautological line bundles
\begin{align}
    \xymatrix{
\eta_{u_1} \times \cdots \times \eta_{u_n} \ar[d]& \\
\mathbb{P}_\mathbb{R}^{u_1} \times \cdots \times \mathbb{P}_\mathbb{R}^{u_n} \ar[r] & BO(n).
    }
\end{align}

Analogously to Prop. \ref{basis_semi-free_actions} we obtain the following result:

\begin{proposition} \label{basis_Z2-actions}
    Any element in ${\Omega}^{O,\mathbb{Z}_2}_{*}$ can be written as a finite sum of products 
    \begin{align}
        [N_0] \times [\mathbb{P}_\mathbb{R}(\eta_{u_1} \times \cdots \times \eta_{u_{k'}}  \oplus \mathbb{R})]
    \end{align}
    where $[N_0] \in \Omega^{O}_{*}$ is a manifold with trivial $\mathbb{Z}_2$-action.
\end{proposition}

Now that we have explicit bases for the bordism groups $\overline{\Omega}^{U,T}_{2*}$ and ${\Omega}^{O,\mathbb{Z}_2}_{*}$ comprised by complex and real algebraic varieties presented in Prop. \ref{basis_semi-free_actions} and Prop. \ref{basis_Z2-actions} respectively, we  just need to check that the set of real points of the complex algebraic varieties that generate the bordism group of semi-free actions spans the basis of real algebraic varieties of the $\mathbb{Z}_2$-equivariant unoriented bordism group.

The real points of the bundle $\gamma_s \to \mathbb{P}_\mathbb{C}^s$ is the bundle $\eta_s \to \mathbb{P}_\mathbb{R}^s $, and therefore the real points of the twisted projective bundle:
\begin{align}
    \mathbb{P}^t_\mathbb{C}(\gamma_{s_1}^+ \times \cdots \times \gamma_{s_{k^+}}^+ \times \gamma_{r_1}^- \times \cdots \times \gamma_{r_{k^-}}^- \oplus \mathbb{C})
\end{align}
is the real projective bundle:
\begin{align}
\mathbb{P}_\mathbb{R}(\eta_{s_1} \times \cdots \times \eta_{s_{k^+}} \times \eta_{r_1} \times \cdots \times \eta_{r_{k^-}} \oplus \mathbb{R})
\end{align}
where the induced $\mathbb{Z}_2$-action reduces to the action:
\begin{align}
t \cdot [\nu_1: \dots : \nu_{{k^+}}:\nu'_1: \cdots : \nu'_{k^-}:x] = [t\nu_1: \dots : t\nu_{{k^+}}:t\nu'_1: \cdots : t\nu'_{k^-}:x] 
\end{align}
for $t = \pm 1$.

\subsection{Equivariant Milnor map}

We now have all the necessary ingredients to deduce the existence of the Milnor map for semi-free actions based on  Prop. \ref{basis_semi-free_actions} and Prop. \ref{basis_Z2-actions}:

\begin{theorem} \label{Milnor_map_semi-free}
    There exists an equivariant Milnor map defined on generators:
    \begin{align}
        \mu^T : \overline{\Omega}^{U,T}_{2*} & \to \Omega^{O,\mathbb{Z}_2}_* \\ \nonumber
    [\mathbb{P}^t_\mathbb{C}(\gamma_{s_1}^+ \times \cdots \times \gamma_{s_{k^+}}^+ \times \gamma_{r_1}^- \times \cdots \times \gamma_{r_{k^-}}^- \oplus \mathbb{C})] & \mapsto    
    [\mathbb{P}_\mathbb{R}(\eta_{s_1} \times \cdots \times \eta_{s_{k^+}} \times \eta_{r_1} \times \cdots \times \eta_{r_{k^-}} \oplus \mathbb{R})]\\ \nonumber
   [ \mathbb{P}_\mathbb{C}^s], \ [H_\mathbb{C}(m,n)] &\mapsto [\mathbb{P}_\mathbb{R}^s], \ [H_\mathbb{R}(m,n)].
    \end{align}
   This map is moreover is surjective.
\end{theorem}

We apply Prop. \ref{generalization_M=FxF} to the previous result
and we obtain a generalization of Milnor's result \cite[Thm. 1]{Milnor-SW}:

\begin{corollary}
    A $\mathbb{Z}_2$-equivariant unoriented cobordism class contains a complex manifold 
    with semi-free circle action if and only if
    it contains a square $N \times N$.
\end{corollary}

\begin{proof}
 The surjectivity of the Milnor map of Thm.
\ref{Milnor_map_semi-free} implies that any square comes from a complex manifold with semi-free circle action. Prop. \ref{generalization_M=FxF} implies that
the $\mathbb{Z}_2$-equivariant unoriented class of a complex manifold with semi-free circle action is cobordant to the square of its real set. We have therefore the commutativity of the following diagram:
\begin{align}
    \xymatrix{
    \overline{\Omega}^{U,T}_{2*} \ar[rr]^F \ar@{->>}[rd]^{\mu^T}&& {\Omega}^{O,\mathbb{Z}_2}_{2*}\\
    & {\Omega}^{O,\mathbb{Z}_2}_{*}
    \ar[ur]^{\wedge 2}&
    }
\end{align}
where $F$ is the forgetful map that sends a unitary manifold with semi-free circle action $[M]$ to the underlying unoriented bordism class $[M]_2$ with $\mathbb{Z}_2$-action, the map $\wedge 2$ sends a $\mathbb{Z}_2$-equivariant unoriented bordism class to its square $[N]_2 \mapsto [N \times N]_2$, and $\mu^T$ is the equivariant Milnor map.
\end{proof}

One notices that there is an intermediate bordism ring between the bordism ring of unitary semi-free actions and the $\mathbb{Z}_2$-equivariant unoriented bordism ring. This is the $\mathbb{Z}_2$-equivariant unitary bordism $\Omega^{U, \mathbb{Z}_2}_*$ of unitary involutions. Moreover, this ring sits in the middle of the ring homomorphisms:
\begin{align}
     \overline{\Omega}^{U,T}_{2*} \to  {\Omega}^{U,\mathbb{Z}_2}_{2*} \to {\Omega}^{O,\mathbb{Z}_2}_{2*},
\end{align}
where the first map restricts the action to the subgroup $\mathbb{Z}_2 \subset T$ and the second map forgets the unitary structure.

A similar argument as the one given in Prop. \ref{basis_semi-free_actions}
permits to show that any element of the the bordism ring
$ {\Omega}^{U,\mathbb{Z}_2}_{2*}$ may be written as a finite sum of projective spaces 
\begin{align}
        [M_0] \times [\mathbb{P}_\mathbb{C}(\gamma_{s_1} \times \cdots \times \gamma_{s_{k}}  \oplus \mathbb{C})]
\end{align}
where $[M_0] \in \Omega^{U}_{2*}$ is a unitary manifold with trivial $\mathbb{Z}_2$-action, and the projective spaces are defined as in Eqn. \eqref{definition_bundle_over_CP}.

\begin{corollary}
\label{Milnor_map_Z2}
    The $\mathbb{Z}_2$-equivariant Milnor map:
 \begin{align} \label{equivariant_Milnor-Map}
        \mu^{\mathbb{Z}_2} : {\Omega}^{U,\mathbb{Z}_2}_{2*} & \to \Omega^{O,\mathbb{Z}_2}_* \\ \nonumber
    [\mathbb{P}_\mathbb{C}(\gamma_{s_1} \times \cdots \times \gamma_{s_{k}}\oplus \mathbb{C})] & \mapsto    
    [\mathbb{P}_\mathbb{R}(\eta_{s_1} \times \cdots \times \eta_{s_{k}}  \oplus \mathbb{R})]\\ \nonumber
   [ \mathbb{P}_\mathbb{C}^s], \ [H_\mathbb{C}(m,n)] &\mapsto [\mathbb{P}_\mathbb{R}^s], \ [H_\mathbb{R}(m,n)],
    \end{align}
 is also surjective.
\end{corollary}
 \begin{proof}
 This follows from the commutative diagram
\begin{align}
    \xymatrix{    \overline{\Omega}^{U,T}_{2*} \ar[d]_{\mu^T} \ar[r] & {\Omega}^{U,\mathbb{Z}_2}_{2*} \ar[d]^F \ar[dl]_{\mu^{\mathbb{Z}_2}} \\
    {\Omega}^{O,\mathbb{Z}_2}_{*} \ar[r]^{\wedge 2}& 
    {\Omega}^{O,\mathbb{Z}_2}_{2*} 
    }
\end{align}
and Thm. \ref{Milnor_map_semi-free}.
\end{proof}

\begin{remark} 
The equivariant analogue for $G=\mathbb{Z}_2$ was established by Beem and Wheeler~\cite[Thm. 1]{BeemWheeler}. They proved that the image of the reduction homomorphism $\Omega_{2*}^{U,\mathbb{Z}_2}\to \Omega_{2*}^{O,\mathbb{Z}_2}$ coincides with the image of the squaring homomorphism $x\mapsto x^2$ from $\Omega_{*}^{O,\mathbb{Z}_2}$ to itself (while for finite groups of odd order the containment is proper \cite[Thm. 2]{BeemWheeler}).

  In the framework of our paper, their result is recovered and refined as follows. Cor.~\ref{Milnor_map_Z2}
  asserts the surjectivity of the 
  $\mathbb{Z}_2$-equivariant Milnor map $\mu^{\mathbb{Z}_2} : {\Omega}^{U,\mathbb{Z}_2}_{2*} \to \Omega^{O,\mathbb{Z}_2}_*$, and Prop.~\ref{generalization_M=FxF} shows that the $\mathbb{Z}_2$-equivariant unoriented class of the real points $F$ of a stably almost complex $\mathbb{Z}_2$-manifold is cobordant to the square $F\times F$. Together, these two statements give a geometric reformulation of \cite[Thm. 1]{BeemWheeler} in the language of equivariant Milnor maps: a class in $\Omega^{O,\mathbb{Z}_2}_*$ lies in the image of the reduction homomorphism if and only if it is a square. 
  Beyond recovering the result of Beem and Wheeler, our approach yields three further refinements. First, we construct the equivariant Milnor map on explicit generators, rather than working only with the abstract reduction homomorphism. Second, we compute the kernel of $\mu^{\mathbb{Z}_2}$ explicitly (Lem.~\ref{lemma_kernel_muZ2}). Finally, in Section 4, we identify this kernel with the image of magnetic unitary equivariant bordism under free conjugations (Thm.~\ref{theorem_sequences_Milnor_map} and Thm.~\ref{theorem_sequence_Z2_Milnor_map}).
\end{remark}

Since the $\mathbb{Z}_2$-equivariant Milnor map $\mu^{\mathbb{Z}_2}$ is compatible with the Milnor map on the Grassmannians:
\begin{align}
    \mu^k : \Omega^{U}_{2*}(BU(k)) \to \Omega_*^O(BO(k)),
\end{align}
then we obtain the following result.
\begin{lemma} \label{lemma_kernel_muZ2}The kernel of the $\mathbb{Z}_2$-equivariant Milnor map  is the ideal:
\begin{align}
    \mathrm{ker}(\mu^{\mathbb{Z}_2}) = 2{\Omega}^{U,\mathbb{Z}_2}_{2*} +\Omega^{U,\mathbb{Z}_2}_{2*}\langle [H_{\mathbb{C}}(2^{s-1},2^{s-1})] \colon s \geq 1 \rangle.
\end{align}
\end{lemma}
This generalizes the calculation of the kernel of the Milnor map presented in Lem. \ref{lemma_kernel_mu}.
Here, it is worth noting that the Milnor hypersurfaces $H_{\mathbb{C}}(2^{s-1},2^{s-1})$ are taken as unitary manifolds with trivial $\mathbb{Z}_2$-action.

\section{Actions with isolated fixed points }


As seen, a key point for the existence of an equivariant Milnor map from 
$\Omega_{2*}^{U, T^k}$ to   $\Omega_*^{O, \mathbb{Z}_2^k}$ depends upon whether there is an involution on a representative of each class in $\Omega_{2*}^{U, T^k}$ such that its fixed point set can admit a $\mathbb{Z}_2^k$-action, representing a class in $\Omega_*^{O, \mathbb{Z}_2^k}$. 

The seminal work of Davis and Januszkiewicz \cite{Davis-Januszkiewicz} 
provides the possibility for further work. Davis and Januszkiewicz introduced and studied two kinds of equivariant manifolds, quasitoric manifolds and small covers, both of which are topological versions of toric varieties. These are characterized as follows.

Let $d = 1$ if $\mathbb{F} = \mathbb{R}$, and $d = 2$ if $\mathbb{F} = \mathbb{C}$.
A $G_{\mathbb{F}}^n$-manifold $M^{dn}$ over a simple $n$-polytope $P^n$ is a smooth closed manifold (which is required to be oriented if $d = 2$) with a locally standard $G_{\mathbb{F}}^n$-action such that its orbit space is homeomorphic to $P^n$. Here, $G_{\mathbb{F}}^n$ is 
$\mathbb{Z}_2^n$ if $\mathbb{F} = \mathbb{R}$, and $T^n$ if $\mathbb{F} = \mathbb{C}$, and 
a locally standard action means that the $G_{\mathbb{F}}^n$-action is locally isomorphic to a standard faithful representation of $G_{\mathbb{F}}^n$ on $\mathbb{F}^n$.
Such a $G_{\mathbb{F}}^n$-manifold is called a \emph{small cover} if $\mathbb{F} = \mathbb{R}$ and a \emph{quasitoric manifold} if $\mathbb{F} = \mathbb{C}$. It was shown in \cite{Buchstaber-Panov-Ray-polytopes} that every quasitoric manifold always admits 
a unitary structure, which  depends on a choice of its omniorientations.

A $G_{\mathbb{F}}^n$-manifold $\pi: M^{dn} \to P^n$ determines a \emph{characteristic function} $\lambda_{\mathbb{F}}: \mathcal{F}(P^n) \to  \text{Hom}(G_{\mathbb{F}}, G_{\mathbb{F}}^n)$, satisfying that for each vertex $v = F_{i_1} \cap \dots \cap F_{i_n}$ of $P^n$, the elements $\lambda_{\mathbb{F}}(F_{i_1}), \dots, \lambda_{\mathbb{F}}(F_{i_n})$ form a basis of $\text{Hom}(G_{\mathbb{F}}, G_{\mathbb{F}}^n)$, where  $\mathcal{F}(P^n) = \{F_1, \dots, F_m\}$ is the set of all facets of $P^n$.
Note that $\text{Hom}(G_{\mathbb{F}}, G_{\mathbb{F}}^n)$ is isomorphic to $\mathbb{Z}_2^n$ if $\mathbb{F} = \mathbb{R}$ and $\mathbb{Z}^n$ if $\mathbb{F} = \mathbb{C}$.

Davis and Januszkiewicz \cite{Davis-Januszkiewicz} tell us that one can use the principal $G_{\mathbb{F}}^n$-bundle 
$P^n\times G_{\mathbb{F}}^n$ over $P^n$
to reconstruct $M^{dn}$ via $\lambda_{\mathbb{F}}$. First $\lambda_{\mathbb{F}}$ gives the following  equivalence relation
$\sim_{\lambda_\mathbb{F}}$ on $P^n\times G_{\mathbb{F}}^n$
\begin{equation}\label{equiv}
(x, g)\sim_{\lambda_{\mathbb{F}}} (y, h)\Longleftrightarrow \begin{cases} x=y, g=h & \text{ if } x\in \text{\rm int}(P^n)\\
x=y, g^{-1}h\in G_F &\text{ if } x\in \text{\rm int}F\subset
\partial P^n
\end{cases}\end{equation}
where $G_F$ is explained as follows:  for each point $x\in
\partial P^n$, there exists a unique face $F$ of $P^n$ such that $x$
is in its relative interior. If $\dim F=k$, then there are $n-k$
facets, say $F_{i_1}, ..., F_{i_{n-k}}$, such that
$F=F_{i_1}\cap\cdots \cap F_{i_{n-k}}$, and furthermore,
$\lambda_{\mathbb{F}}(F_{i_1}), ..., \lambda_{\mathbb{F}}(F_{i_{n-k}})$ determine a
subgroup of rank $n-k$ in $G_{\mathbb{F}}^n$, denoted by $G_F$. Then the reconstruction of $M^{dn}$ is exactly
the quotient space $P^n\times G_{\mathbb{F}}^n/\sim_{\lambda_{\mathbb{F}}}$,  denoted
by $M(P^n, \lambda_{\mathbb{F}})$. 

It is clear that the complex conjugation on $T^n\subset \mathbb{C}^n$ induces an involution $\tau$
on $M(P^n, \lambda_\mathbb{C})$ such that the fixed point set 
$M(P^n, \lambda_\mathbb{C})^\tau$
is just the small cover $M(P^n, \lambda_\mathbb{R})$, where $\lambda_\mathbb{R}$ is the mod 2 reduction of
$\lambda_\mathbb{C}$. On the other hand, a natural question arises: {\em can every small cover be obtained in such a way?}
This is equivalent to the following {\em lifting problem}: Suppose that there exists a characteristic function 
$\lambda_{\mathbb{R}}: \mathcal{F}(P^n) \to  \mathbb{Z}_2^n$. Does there exist a characteristic function 
$\lambda_{\mathbb{C}}: \mathcal{F}(P^n) \to  \mathbb{Z}^n$ such that the following diagram commutes?
\begin{equation*} \xymatrix{
& \mathbb{Z}^n \ar[d]^{\mod 2}\\
\mathcal{F}(P^n) \ar[ur]^{\lambda_{\mathbb{C}}}\ar[r]_{\lambda_{\mathbb{R}}} & \mathbb{Z}_2^n.
}\end{equation*}
It is worth pointing out that
some related work has been carried out in the papers \cite{Qifan, Choi-Park, Sun, Choi-Yang-Vallee},  among others. 

\begin{definition}
Consider the subgroup 
$\widehat{\Omega}_{2n}^{U, T^n}$ (resp. $\widehat{\Omega}_{n}^{O, \mathbb{Z}_2^n}$) of $\Omega_{2n}^{U, T^n}$ (resp. $\Omega_{n}^{O, \mathbb{Z}_2^n}$), where 
$\widehat{\Omega}_{2n}^{U, T^n}$ is generated by all equivariant unitary bordism classes of $2n$-dimensional unitary  manifolds with  fully effective  $T^n$-actions, and $\widehat{\Omega}_{n}^{O, \mathbb{Z}_2^n}$ is generated by equivariant unoriented bordism classes of all $n$-dimensional  manifolds with fully effective  $\mathbb{Z}_2^n$-actions.
Here, an action of $T^n$ or $\mathbb{Z}_2^n$ on a manifold is fully effective if $T^n$ or $\mathbb{Z}_2^n$ acts effectively on each connected component, and the effectiveness of an action means that no non-trivial element of $T^n$ or $\mathbb{Z}_2^n$ fixes every point in the manifold. In particular, the effectiveness implies that the fixed point set is either empty or consists of isolated points in these two cases.
\end{definition}


It was shown by the third author and collaborators \cite{LuTan, LuChenTan}   the following results:
\begin{lemma}\label{representative}\
\begin{itemize}
 \item[(1)] 
 Each class in $\widehat{\Omega}_{n}^{O, \mathbb{Z}_2^n}$ contains a small cover as its representative.
 \item[(2)] 
 For $n>1$, each class of $\widehat{\Omega}_{2n}^{U, T^n}$ contains an omnioriented quasitoric manifold as its representative.
\end{itemize}
\end{lemma}

Note that for $n=1$, $\widehat{\Omega}_{2}^{U, T}\cong \mathbb{Z}$ is exactly generated by the class of $\mathbb{C}P^1$ with the standard $T$-action, while $\widehat{\Omega}_1^{O,\mathbb{Z}_2}\cong \{0\}$.

Furthermore, \cite{LuChenTan} demonstrates that the groups $\widehat{\Omega}_{n}^{O, \mathbb{Z}_2^n}$ and $\widehat{\Omega}_{2n}^{U, T^n}$ are generated by certain special types of small covers and omnioriented quasitoric manifolds, respectively.

\begin{lemma}[\cite{LuChenTan}]\label{generation}\
\begin{itemize}
 \item[(1)] As a finite-dimensional vector space over $\mathbb{Z}_2$, $\widehat{\Omega}_{n}^{O, \mathbb{Z}_2^n}$ has a basis consisting of  certain equivariant unoriented bordism classes of $n$-dimensional generalized real Bott manifolds (i.e., small covers over products of some simplices). 
 \item[(2)] As a free abelian $\mathbb{Z}$-module, $\widehat{\Omega}_{2n}^{U, T^n}$ has a basis given by certain equivariant unitary bordism classes of $2n$-dimensional omnioriented quasitoric manifolds over products of some simplices.
\end{itemize}
\end{lemma}

Now with these preliminaries, we can define a group homomorphism 
\begin{align*} 
    \widehat{\mu}^{T^n} : \widehat{\Omega}_{2n}^{U, T^n} \longrightarrow 
  \widehat{\Omega}_{n}^{O, \mathbb{Z}_2^n}
\end{align*}
as follows. Choose a basis $S$ of $\widehat{\Omega}_{2n}^{U, T^n}$ as described in Lem. \ref{generation} (2). For any $[M(P^n, \lambda_\mathbb{C})]\in S$, set $\widehat{\mu}^{T^n}([M(P^n, \lambda_\mathbb{C})]) = [M(P^n,\lambda_\mathbb{R})]_2$. From the corresponding fixed point data, which is a formal sum of all tangent representations at fixed points \cite{Stong_1970, LuTan, Darby}, the image of $[M(P^n, \lambda_\mathbb{C})]$ under the map $\widehat{\mu}^{T^n}$ is independent of the choice of omnioriented quasitoric manifolds. 
Then extend $\widehat{\mu}^{T^n}$ linearly to the whole group $\widehat{\Omega}_{2n}^{U, T^n}$.

For an arbitrary element $x = [M(P^n, \lambda_\mathbb{C})]$ in $\widehat{\Omega}_{2n}^{U, T^n}$, write $x = \sum a_i s_i$ with $s_i\in S$ and $a_i\in \mathbb{Z}$. Then $\widehat{\mu}^{T^n}(x) = \sum a_i \widehat{\mu}^{T^n}(s_i)$. 
As can be seen from the fixed point data, this image is exactly $[M(P^n, \lambda_\mathbb{R})]_2$.

Obviously, if the lifting problem holds, then $\widehat{\mu}^{T^n}$ would be surjective. However, the problem remains open except for some special cases. For example,  $|\mathcal{F}(P^n)| - n = 3, 4$ \cite{Choi-Park, Choi-Yang-Vallee}. Nevertheless, we can still prove that $\widehat{\mu}^{T^n}$ is surjective for the following reason.
If $P^n$ is a product of some simplices, it is well-known that the lifting problem always holds (see Example 2.7 in \cite{Choi-Yang-Vallee}). Therefore, together with Lemma~\ref{generation} (1), we have the following result:

\begin{theorem}\label{Milnor_map_fully_effective}
  $ \widehat{\mu}^{T^n} : \widehat{\Omega}_{2n}^{U, T^n} \longrightarrow 
  \widehat{\Omega}_{n}^{O, \mathbb{Z}_2^n}$  is a surjective homomorphism.  
\end{theorem}

Theorem~\ref{Milnor_map_fully_effective} indicates that the lifting problem holds in the sense of equivariant bordism. This conclusion was already proved in Tan's thesis \cite[Thm. 5.1]{Tan_thesis}. We revisit it here in order to present another equivariant version of Milnor map and keep the exposition self-contained.

\section{Magnetic unitary bordism}

Almost complex manifolds and stably almost complex manifolds with compatible conjugation maps have been the subject of study in several authors \cite{Landweber_fixed_pont_free_conjugations,  Landweber_conjugation_MU, Stong_manifolds_with_conjugation, Stong_error_conjugation, Edelson_conjugations, Edelson_fixed_of_conjugations}. Here we propose a bordism theory which incorporates both, the stably unitary 
structure, and the compatible conjugation involution. This defines the bordism theory associated to magnetic equivariant vector bundles \cite{serrano2025magneticequivariantktheory}.

\subsection{Magnetic equivariant vector bundles}

A {\it magnetic group} is a group $G$ endowed with a $\mathbb{Z}_2$ grading defined by a surjective homomorphism $\phi : G \to \overline{\mathbb{Z}}_2$. Here  $\overline{\mathbb{Z}}_2$ will simply be the group ${\mathbb{Z}}_2$, the line over it has the purpose of distinguishing it from other $\mathbb{Z}_2$ groups that may appear in $G$.    A {\it magnetic representation} (also known as a corepresentation in Wigner's book \cite{Wigner_book}) is a complex vector space $V$ where the group $G$ acts such that the elements of $G_0 = \mathrm{ker}(\phi)$ act complex linearly and the elements of $G \backslash G_0$ act complex antilinearly. Morphisms of magnetic representations are defined by $G$-equivariant complex linear maps.

If $G$ is a compact Lie group, the irreducible magnetic representations
of $G$ can be classified by its algebra of $G$-equivariant endomorphisms. This can also be seen 
by the type of complex irreducible representations that appear once the magnetic representation is restricted to the subgroup $G_0$. If $V$ is a magnetic irreducible representation and $\mathrm{res}^G_{G_0}$ denotes the restriction to $G_0$, then Wigner proved:
\begin{align}
    \mathrm{End}_G(V) \cong\begin{cases}
        \mathbb{R} & \mathrm{res}^G_{G_0}V \in \mathrm{Irrep}(G_0)\\
         \mathbb{H} & \mathrm{res}^G_{G_0}V \cong W \oplus W \\
          \mathbb{C} & \mathrm{res}^G_{G_0}V \cong W \oplus \overline{W}
    \end{cases}
\end{align}
where $W \in \mathrm{Irrep}(G_0)$ and $\overline{W}$ is its conjugate representation (taking into account the complex conjugation and the conjugation action of $\overline{\mathbb{Z}}_2$ in $\mathrm{Irrep}(G_0)$).

If $X$ is a finite $G$-equivariant $CW$-complex, a magnetic $G$-equivariant vector bundle over $X$
is a complex vector bundle $E$ endowed with a compatible action of $G$ such that it is complex linear on the fibers for $G_0$ and complex antilinear for $G \backslash G_0$.

The Grothendieck group of isomorphism classes of magnetic $G$-equivariant vector bundles over $X$ defines the {\it magnetic $G$-equivariant K-theory} $\mathbf{K}_G(X)$ and the suspensions $$\widetilde{\mathbf{K}}_G^{-q}(X) \coloneqq \widetilde{\mathbf{K}}_G(\Sigma^qX)$$
and Bott periodicity ${\mathbf{K}}_G^{-q} \cong {\mathbf{K}}_G^{-q-8}$ permits to define a $G$-equivariant cohomology theory ${\mathbf{K}}_G^{*}$.

This K-theory comes with a natural forgetful map to complex $G_0$-equivariant K-theory $\mathbf{K}_G^*(X) \to KU^*_{G_0}(X)$, landing in the $\overline{\mathbb{Z}}_2$-invariant part and becoming a rational isomorphism \cite{Serrano_rationalmagnetic}:
\begin{align}
    \mathbf{K}_G^*(X) \otimes \mathbb{Q} \overset{\cong}{\to} KU^*_{G_0}(X)^{\overline{\mathbb{Z}}_2} \otimes \mathbb{Q}.
\end{align}

\subsection{Magnetic equivariant bordism}

The {\it magnetic $G$-equivariant bordism} is the bordism theory framed on stably magnetic $G$-equivariant bundles.
An $n$-dimensional manifold $M$ endowed with a $G$-action is said to be {\it  stably magnetic $G$-equivariant} if the stabilization of the tangent bundle can be endowed with the structure of a magnetic $G$-equivariant vector bundle. This is an isomorphism of real $G$-equivariant bundles
\begin{align}
    TM^n \oplus \theta \cong \xi^N_\mathbb{C}
\end{align}
where $\xi_\mathbb{C}$ is a magnetic $G$-equivariant vector bundle over $M$ of complex rank $N$, $\theta$ is a  trivial real vector bundle of rank $2N-n$ over $M$ endowed with an involution of real vector bundles (or a $\mathbb{Z}_2$-action) $\mu : \theta \to \theta$ compatible with the $G$-action on $M$, taken as a $G$-equivariant real vector bundle through the homomorphism $\phi$. This means that $G$ acts on $\theta \simeq M \times \mathbb{R}^{2N-m}$  by the equation 
\begin{align}
g(m,\lambda) = \begin{cases} (gm, \mu(\lambda)) &  \phi(g)=1\\ (gm,\lambda) & \phi(g)=0.
\end{cases}
\end{align}
Two stable magnetic $G$-equivariant structures $ TM_i \oplus \theta_i \cong (\xi_i)_\mathbb{C}$, $i =1,2$, are {\it stably equivalent} if there is a
$G$-equivariant diffeomorphism $f: M_1 \overset{\cong}{\to} M_2$
and an isomorphism of magnetic $G$-equivariant  bundles
\begin{align}
     TM_1 \oplus \theta_1 \oplus \mathbb{C}^{k_1} \cong (\xi_1)_\mathbb{C} \oplus \mathbb{C}^{k_1} \cong (\xi_2)_\mathbb{C} \oplus \mathbb{C}^{k_2} \cong  TM_2 \oplus \theta_2 \oplus \mathbb{C}^{k_2} 
\end{align}
for some $k_1$ and $k_2$,
compatible with the differential $df_*:TM_1 \to TM_2$. Here, the magnetic $G$-equivariant structure on the extra copies of $\mathbb{C}$ is given by the homomorphism $\phi : G \to \overline{\mathbb{Z}}_2$; namely complex conjugation via $\phi$.

In terms of framings it is equivalent to find a lift:
\begin{align}
    \xymatrix{
&& \mathbb{B}_GU(N) \ar[d] \ar[r] &\mathbb{B}_GU \ar[d] \\
M \ar[r]_h \ar@{.>}[rru]^\xi& B_GO(n) \ar[r]_{\times \theta} & B_GO(2N) \ar[r] &B_GO    
    }
\end{align}
where $B_GO(n)$ is the classifying space for $G$-equivariant real bundles, $h$ denotes the classifying map of the tangent bundle, and 
$\mathbb{B}_GU(N)$ is the classifying space of magnetic $G$-equivariant vector bundles. This last space can be taken to be the Grassmannian of $N$-dimensional complex planes in
a complete universe for magnetic $G$-equivariant representations. The spaces $\mathbb{B}_GU$ and $B_GO$ denote the limit of the stabilizations.

The manifold $M$ bounds in this context if there is a magnetic $G$-equivariant manifold $W$ such that $\partial W=M$. In other words, there exists a lift $W \to \mathbb{B}_GU$ of the tangent bundle $W \to B_GO$ such that the restriction map $\partial W\to \mathbb{B}_GU$ is $G$-homotopic to  the map $\xi$. Note that the restriction of a stably magnetic $G$-equivariant bundle on $W$ to its boundary $M$:
\begin{align}
    (TW \oplus \theta )|_M \cong TM \oplus \mathbb{R} \oplus \theta|_M
\end{align}
is also a stably magnetic $G$-equivariant bundle. Note that the normal bundle of $M$ in $W$ 
is trivial and the elements in $G \backslash G_0$ act trivially on the fibers. The only non-trivial action comes from the group $\overline{\mathbb{Z}}_2$ and its action can only be trivial or the sign representation.

\begin{definition}
    The {\bf magnetic $G$-equivariant bordism group} consists of equivalence classes of manifolds endowed with stably magnetic $G$-equivariant tangent bundle up to the bordism relation. We denote these groups by $\mathbf{\Omega}^G_*$. For a $G$-space $X$ the bordisms may be anchored by $G$-equivariant maps to $X$. Denote by $\mathbf{\Omega}^G_*(X)$ the associated bordism classes. 
\end{definition}

These bordism groups $\mathbf{\Omega}^G_*(X)$ define a $G$-equivariant homology theory following the general structure defined by Conner and Floyd \cite{Conner_Floyd_Differentiable}.

Note that the magnetic bordism groups come with natural forgetful maps:
\begin{align}
  \Omega^{O,G}_* \leftarrow  \mathbf{\Omega}^G_* \to \Omega^{U,G_0}_*
\end{align}
where the first forgets the unitary structure, and the second restricts the bordisms to a $G_0$-action. 

\subsection{Magnetic unitary bordism} We are particularly interested in the magnetic bordism groups $\mathbf{\Omega}^{\overline{\mathbb{Z}}_2}_*$ and $\mathbf{\Omega}^{T^k \rtimes \overline{\mathbb{Z}}_2}_*$. Here, the magnetic structure of the groups is given by the projection of the groups to $\overline{\mathbb{Z}}_2$, and the action of $\overline{\mathbb{Z}}_2$ on the torus $T^k$ is given by complex conjugation.

Before analyzing these bordism groups let us briefly mention some of the previous related results.  Landweber  was the first to study the $\mathbb{Z}_2$-equivariant homotopy of the unitary bordism spectrum $MU$ defining
the homotopy bordism groups \cite{Landweber_conjugation_MU}:
\begin{align}
MU_{p,q} = \pi_{p,q}(MU) = \lim_k \pi_{p+k,q+k}(MU(k)) =\lim_k [(B^{p+k,q+k}, \partial B^{p+k,q+k}), (MU(k),*)]_{\mathbb{Z}_2}
\end{align}
where $B^{p,q}$ is the unit   in $\mathbb{R}^{p,q}$, its $\mathbb{Z}_2$-structure maps $(x,y)$ to $(-x,y)$, $MU(k)$ is the Thom space of the canonical bundle over $BU(k)$ with the natural conjugation map, and $[-,-]_{\mathbb{Z}_2}$ denotes the homotopy classes of $\mathbb{Z}_2$-equivariant maps. There are canonical maps:
\begin{align}
    \Omega^{O}_q \leftarrow MU_{p,q} \to \Omega^U_{p+q}
\end{align}
where the first restricts to the $\mathbb{Z}_2$ fixed points, and the second forgets the $\mathbb{Z}_2$-action. Landweber further shows that the following diagram is commutative:
\begin{align}
\xymatrix{
    MU_{n,n} \ar@{->>}[r] \ar@{->>}[d] & \Omega^U_{2n} \ar@{->>}[ld]^\mu \\
    \Omega^O_n }
\end{align}
where both the horizontal and the vertical maps are surjective and $\mu$ denotes the Milnor map. Moreover, the horizontal map is a rational isomorphism and its kernel is 2-torsion.

Our interest lies in the geometric version of Landweber bordism groups. Allow us to make the following definition:

\begin{definition}
    Call  {\bf magnetic unitary bordism ring}  the magnetic $\overline{\mathbb{Z}}_2$-equivariant bordism ring $\mathbf{\Omega}^{\overline{\mathbb{Z}}_2}_*$.
\end{definition}

A magnetic unitary manifold consists of the following information:
\begin{itemize}
\item A manifold $M$ endowed with an involution $\tau : M \to M$.
    \item A stably almost complex structure $J : TM \oplus \theta \to TM \oplus \theta  $ where $\theta$ is a trivial real bundle over $M$.
    \item A lift of the involution
    $\tau_* : TM \oplus \theta \to TM \oplus \theta$ such that $\tau_* \circ J = - J \circ \tau_*$ and $\tau_*^2=1.$
\end{itemize}

Although the three maps $\tau$, $J$ and $\tau_*$ are essential, we will identify the magnetic unitary class of this manifold with the pair $(M, \tau_*)$; the stably almost complex structure will be clear from the context,

The additive inverse $-[(M, \tau_*)]$ of the bordism class $[(M, \tau_*)]$
consists of the same manifold $M$ with involution $\tau$, together  with the stably almost complex structure $J \oplus (-i) : TM \oplus \theta \oplus \mathbb{C} \to  TM \oplus \theta \oplus \mathbb{C}$, and involution
$\tau_* \oplus \mathbb{K}$. This follows from taking  the boundary of the interval $[-1,1]$ with constant stably almost complex structure $T[-1,1] \times \mathbb{R} \cong [-1,1]  \times \mathbb{C} $ and conjugation $\mathbb{K}$. The boundary consists of two points with opposite complex structures $\mathbb{C} \sqcup \overline{\mathbb{C}}=\{1\} \times \mathbb{C} \cup \{-1\} \times \overline{\mathbb{C}}$.

\begin{remark}
Stong in \cite{Stong_manifolds_with_conjugation} defined a bordism theory for stably 
magnetic $\overline{\mathbb{Z}}_2$-equivariant normal bundles where it 
was denoted $\Omega^{AR}_*$; here the AR is for Atiyah's real vector bundles.
These bordism groups are the same as the magnetic unitary bordism groups $\mathbf{\Omega}^{\overline{\mathbb{Z}}_2}_*$ and it was proven by Stong himself in \cite[p. 341]{Stong_manifolds_with_conjugation}.

We prefer to restrict our definition of the magnetic 
equivariant bordism groups
through the stable information associated to the tangent bundle. This choice is aligned with the definition of the unitary equivariant bordism groups 
done in \cite[\S 2]{Hankes}
and that has been used by the last author \cite{AngelGomezUribe, Uribe_evenness, AngelSampertonSegoviaUribe} and others \cite{Sinha-semi-free}.  Note that the
restriction of a stably magnetic $G$-equivariant vector bundle $TM \oplus \theta$ to the subgroup $G_0$ is  simply  $TM \oplus \mathbb{R}^{N-2n}$
since $G_0$ acts trivially on $\theta$. 
This is the reason why the definitions
 of the unitary equivariant bordism groups and the magnetic equivariant bordism groups are compatible.

Some of the properties of the bordism groups  $\Omega^{AR}_*$ also hold in the magnetic bordism $\mathbf{\Omega}^{\overline{\mathbb{Z}}_2}_*$. Here we will explain how some of the proofs of Stong \cite{Stong_manifolds_with_conjugation} work in our setup.
We must remark however, that a non-trivial mistake was found in the main theorem of Stong's paper \cite{Stong_manifolds_with_conjugation},  and the author himself asked to disregard the whole contents of his work \cite{Stong_error_conjugation}. 
Here we will reestablish some of the results of Stong's paper in the context of magnetic equivariant bordism $\mathbf{\Omega}^{\overline{\mathbb{Z}}_2}_*$, and we will add some of our own.
\end{remark}

Noting that the generators
of the unitary bordism ring can be taken to be complex manifolds, we obtain the following simple result:
\begin{lemma}
    The forgetful homomorphism
    \begin{align} \label{forgetful_map_FU}
F_U:\mathbf{\Omega}^{\overline{\mathbb{Z}}_2}_* \to \Omega^{U}_* 
    \end{align}
    that takes a magnetic unitary bordism class to the unitary bordism class,
    is a split homomorphism of rings.
\end{lemma}
\begin{proof}
    The Milnor hypersurfaces $H_\mathbb{C}(m,n)$, $1 \leq m \leq n$, together with the complex projective spaces $\mathbb{P}_\mathbb{C}^{s}$ generate the unitary bordism ring. All these manifolds are smooth complex projective algebraic varieties and can be endowed with the conjugation action $\mathbb{K}$ coming from the conjugation of  the ground field $\mathbb{C}$. Denote them by the pairs $(H_\mathbb{C}(m,n), \mathbb{K})$ and $(\mathbb{P}_\mathbb{C}^{s}, \mathbb{K})$ and take them as representatives of classes in the magnetic unitary bordism group. Therefore, the assignment
    \begin{align}
        \Omega^{U}_*& \to \mathbf{\Omega}^{\overline{\mathbb{Z}}_2}_* \\
        [H_\mathbb{C}(m,n)] & \mapsto [(H_\mathbb{C}(m,n), \mathbb{K} )]\\
        [\mathbb{P}_\mathbb{C}^{s}] & \mapsto [(\mathbb{P}_\mathbb{C}^{s}, \mathbb{K})]
    \end{align}
    provides the ring splitting of the forgetful homomorphism $F_U$.
\end{proof}

To study the magnetic unitary bordism groups we will make use of the long exact sequence associated to the pair of families  $\mathcal{A}$ and $\{\{1\}\}$ of subgroups of $\overline{\mathbb{Z}}_2$, where the first one has all subgroups $\mathcal{A} = \{ \overline{\mathbb{Z}}_2,\{1\}\}$, and the second one only the trivial one. This has been masterfully explained in \cite[\S 28]{Conner_Floyd_Differentiable}  and defines the long exact sequence:
\begin{align}\label{long_exact_sequence_restricted_unrestricted}
   \cdots \to \mathbf{\Omega}^{\overline{\mathbb{Z}}_2}_*\{\{1\}\} \overset{\iota}{\to} \mathbf{\Omega}^{\overline{\mathbb{Z}}_2}_* \overset{\Psi}{\to} \mathbf{\Omega}^{\overline{\mathbb{Z}}_2}_*\{\mathcal{A}, \{1\}\} \overset{\partial}{\to} \mathbf{\Omega}^{\overline{\mathbb{Z}}_2}_{*-1}\{\{1\}\} \to \cdots
\end{align}
where $\mathbf{\Omega}^{\overline{\mathbb{Z}}_2}_*\{\{1\}\}$ consists of the magnetic unitary bordism group of manifolds whose conjugation is free, and 
$\mathbf{\Omega}^{\overline{\mathbb{Z}}_2}_*\{\mathcal{A}, \{1\}\}$ consists of the magnetic unitary bordism group of manifolds whose boundary has a free action of the conjugation. The map $\iota$ simply sends the magnetic unitary manifolds with free actions to the unrestricted ones. The map $\Psi$ takes a tubular neighborhood of the fixed points of the conjugation map and assigns the unit ball bundle; note that the action of the conjugation is free on the sphere bundle of the normal bundle of the fixed points of the conjugation.

\subsection{Magnetic unitary bordism of free conjugations}
Of particular interest are the magnetic unitary manifolds with free conjugation $\mathbf{\Omega}^{\overline{\mathbb{Z}}_2}_{*}\{\{1\}\}$. Let us describe a series of constructions yielding non-zero stably magnetic manifolds with free conjugations in $\mathbf{\Omega}^{\overline{\mathbb{Z}}_2}_{*}\{\{1\}\}$.

\begin{itemize}
\item {\underline{\it  Twin magnetic unitary manifolds.}} Take a smooth projective manifold $M_{\mathbb{C}}$ defined by homogeneous polynomials with real coefficients  (the generator of the unitary bordism ring $\Omega^U_{2*}$ are of this type). Denote by $\mathbb{K}:M_\mathbb{C} \to M_\mathbb{C}$ the complex conjugation map induced by complex conjugation in the homogeneous coordinates.  Denote the pair $(M_\mathbb{C}, \mathbb{K})$.

Take the disjoint union of two copies of $M$ as complex manifold, 
\begin{align}
    M_\mathbb{C} \sqcup M_\mathbb{C} \coloneqq M \times \{-1\} \cup M \times \{1\},
\end{align}    
    and endow this manifold with the conjugation action:
\begin{align} \label{twin-magnetic-mainfold}
    \mathbb{K} \times(-1) : M_\mathbb{C} \sqcup M_\mathbb{C} &\to M_\mathbb{C} \sqcup M_\mathbb{C} \\
    (m,t) & \mapsto (\mathbb{K}m, -t)
\end{align}
and note that the pair $(M_\mathbb{C}\sqcup M_\mathbb{C},  (\mathbb{K} \times(-1))_*)$ is a magnetic unitary manifold with free conjugation. Call it the {\it twin magnetic unitary manifold}.

The composition  of the homomorphisms of groups
\begin{align}
    \Omega^{U}_{2*}& \longrightarrow  \mathbf{\Omega}^{\overline{\mathbb{Z}}_2}_{2*}\{\{1\}\} \overset{F_U \circ \iota}{\longrightarrow } \Omega^{U}_{2*}\\
    [M_\mathbb{C}] \mapsto & [M_\mathbb{C}\sqcup M_\mathbb{C},  (\mathbb{K} \times(-1))_*]  \mapsto 2[M_\mathbb{C}],
\end{align}
where the second map is the composition of the map $\iota: \mathbf{\Omega}^{\overline{\mathbb{Z}}_2}_{*}\{\{1\}\} \to \mathbf{\Omega}^{\overline{\mathbb{Z}}_2}_{*}$ from magnetic unitary  manifolds with free conjugations to magnetic unitary  manifolds with unrestricted conjugation, and $F_U$ the forgetful map of Eqn \eqref{forgetful_map_FU}, implying that the twin magnetic unitary manifolds 
\begin{align} \label{twin_magnetic_unitary_manifolds}
[M_\mathbb{C}\sqcup M_\mathbb{C},  (\mathbb{K} \times(-1))_*]
\in \mathbf{\Omega}^{\overline{\mathbb{Z}}_2}_{2*}\{\{1\}\}
\end{align}
are all non-trivial.

\item {\underline{\it Opposite twin magnetic unitary manifolds.}} 
For a smooth projective manifold $M_{\mathbb{C}}$ with complex conjugation  $\mathbb{K}: M_\mathbb{C} \to M_\mathbb{C}$  and almost complex structure $J_M$ such that $M_{\mathbb{C}}$ represents a generator of $\Omega^U_{2n}$, we can also take the pair of opposite manifolds:
\begin{align}
    \overline{M}_\mathbb{C} \sqcup {M}_\mathbb{C} \coloneqq \overline{M} \times \{-1\} \cup M \times \{1\},
\end{align}  
where $\overline{M}$ means the same manifold $M$ but with the opposite almost complex structure $-J_M$.

Endow this manifold with the conjugation action:
\begin{align} \label{opposite-twin-magnetic-mainfold}
    \mathrm{id} \times(-1) : \overline{M}_\mathbb{C} \sqcup M_\mathbb{C} &\to \overline{M}_\mathbb{C} \sqcup M_\mathbb{C} \\
    (m,t) & \mapsto (m, -t)
\end{align}
and note that the equation:
\begin{align}
    (-J_M\times\{-1\}) = - (J_M \times \{1\})
\end{align}
implies that the map $\mathrm{id} \times(-1)$ is indeed a conjugation.
The pair $(\overline{M}_\mathbb{C}\sqcup M_\mathbb{C},  (\mathrm{id} \times(-1))_*)$ is a magnetic unitary manifold with free conjugation. Call it the {\it opposite twin magnetic unitary manifold}.

In this case the composition  of the homomorphisms of groups
\begin{align}
    \Omega^{U}_{2n}& \longrightarrow  \mathbf{\Omega}^{\overline{\mathbb{Z}}_2}_{2n}\{\{1\}\} \overset{F_U \circ \iota}{\longrightarrow } \Omega^{U}_{2n}\\
    [M_\mathbb{C}] \mapsto & [\overline{M}_\mathbb{C}\sqcup M_\mathbb{C},  (\mathrm{id} \times(-1))_*]  \mapsto (1+(-1)^n)[M_\mathbb{C}]
\end{align}
depends on the parity of $n$.
In particular, when the complex dimension of $M$ is even, then opposite twin magnetic unitary manifold $(\overline{M}_\mathbb{C}\sqcup M_\mathbb{C},  (\mathrm{id} \times(-1))_*)$ is non-trivial in 
$\mathbf{\Omega}^{\overline{\mathbb{Z}}_2}_{2n}\{\{1\}\}$.

\item {\underline{\it Milnor hypersurfaces with free conjugation.}} The Milnor hypersurfaces $H_\mathbb{C}(m,m)$ defined in Eqn. \eqref{Milnor_hypersurfaces_definition} can be endowed with a free conjugation action as follows:
\begin{align}
    \mathbb{K}_{\mathrm{swap}}: H_\mathbb{C}(m,m) \to H_\mathbb{C}(m,m) \\
    \big([\mathbf{w}], [\mathbf{z}]\big) \mapsto \big([\overline{\mathbf{z}}], [\overline{\mathbf{w}}]\big).
\end{align}

Note that this conjugation is free, because any fixed point of the conjugation would imply that $[\mathbf{w}] = [\overline{\mathbf{z}}]$ and the solution of the equation $z_0\overline{z_0} + \cdots z_m \overline{z_m}=0$ would be empty.

Denote these magnetic unitary manifolds by the pair $(H_\mathbb{C}(m,m), {\mathbb{K}_{\mathrm{swap}}}_*)$ 
and differentiate them from the magnetic unitary manifolds $(H_\mathbb{C}(m,m), \mathbb{K}_*)$ where the conjugation is the complex conjugation and its fixed points are the real Milnor hypersurfaces:
\begin{align}
    H_\mathbb{C}(m,m)^\mathbb{K}= H_\mathbb{R}(m,m).
\end{align}

The magnetic unitary manifolds 
\begin{align} \label{Milnor_hypersurfaces_free}
[H_\mathbb{C}(2^s,2^s) , {\mathbb{K}_{\mathrm{swap}}}_*] \in \mathbf{\Omega}^{\overline{\mathbb{Z}}_2}_{2(2^{s+1}-1)}\{\{1\}\}
\end{align}
define non-trivial elements since their images under the forgetful map $F_U$ are the unitary manifolds $H_\mathbb{C}(2^s,2^s)$ whose characteristic numbers are congruent with 2 modulo 4.
Thus they are generators of the unitary bordism ring modulo 2.

\item {\underline{\it Spheres with antipodal stable conjugation.}}  Take the trivialization of the stabilization of the tangent bundle of the spheres:
\begin{align}
    TS^n \times \mathbb{R} &\overset{\cong}{\to} S^n \times \mathbb{R}^{n+1}\\
    (v,w,r) & \mapsto (v,w+rv)
\end{align}
where $v \in \mathbb{R}^{n+1}$ with $|v|=1$ is a point in $S^n$ and $w \in \mathbb{R}^{n+1}$ with $w \cdot v =0$ represents a point in $T_vS^{n}$. Take the antipodal action $I$ on $\mathbb{R}^{n+1}$ with $I(x)=-x$ and note that the induced action on the stable tangent bundle is:
\begin{align}
  I:  TS^n \times \mathbb{R} \to TS^n \times \mathbb{R}\\
  (v,w,r) \mapsto (-v,-w,r),
\end{align} 
namely, trivial on the normal bundle of the embedding $S^n \hookrightarrow \mathbb{R}^{n+1}$.

Take the stable almost complex structure
\begin{align}
    J_{S^n_{\mathrm{st}}}: TS^n \times \mathbb{R} \times \mathbb{R}^{n+1}   & \to  TS^n \times \mathbb{R} \times \mathbb{R}^{n+1} 
\end{align}
induced from the almost complex structure
\begin{align}
    S^n \times \mathbb{R}^{n+1} \times  \mathbb{R}^{n+1} &\to  S^n \times \mathbb{R}^{n+1} \times  \mathbb{R}^{n+1}\\
    (v,x,y) & \mapsto (v,-y,x).
\end{align}

This almost complex structure can be written as follows:
\begin{align}
    J_{S^n_{\mathrm{st}}}: TS^n \times \mathbb{R} \times \mathbb{R}^{n+1}   & \to  TS^n \times \mathbb{R} \times \mathbb{R}^{n+1} \\
    (v,w,r,w_1+r_1v) & \mapsto (v,-w_1,-r_1,w+rv)
\end{align}
where  $v \in S^n$, $w,w_1 \in T_vS^n$, $r,r_1 \in \mathbb{R}$, and $y=w_1 + r_1v \in \mathbb{R}^{n+1}$ is written
 in terms of the vector $v$ and 
a complementary component in $T_vS^n$.

Define the lift of the antipodal map $I$  as  the map:
\begin{align}
    {I}_* : TS^n \times \mathbb{R} \times \mathbb{R}^{n+1}   & \to  TS^n \times \mathbb{R} \times \mathbb{R}^{n+1} \\
    (v,w,r,y) & \mapsto (-v,-w,r,y)
\end{align}
which in terms of a tangent and normal decompositions it looks like:
\begin{align}
   {I}_* (v,w,r,w_1+r_1v) = (-v,-w,r,w_1-r_1(-v)),
\end{align}
and note that ${I}_* \circ J_{S^n_{\mathrm{st}}} =- J_{S^n_{\mathrm{st}}} \circ {I}_*$ because we have the following equations:
\begin{align}
    {I}_* \circ J_{S^n_{\mathrm{st}}}(v,w,r,w_1+r_1v) & = {I}_*(v,-w_1,-r_1,w+rv) \\
    &  = (-v,w_1,-r_1,w+rv),\\
       J_{S^n_{\mathrm{st}}} \circ  {I}_* (v,w,r,w_1+r_1v) &=
      J_{S^n_{\mathrm{st}}} (-v,-w,r,w_1+r_1v)\\
      &=
      J_{S^n_{\mathrm{st}}} (-v,-w,r,w_1-r_1(-v))\\
       &  = (-v,-w_1,r_1,-w+r(-v))\\
       &  = (-v,-w_1,r_1,-w-rv).
\end{align}

Further note that the involution ${I}_*$ makes the stable tangent bundle of the sphere into the vector bundle:
\begin{align}
    TS^n \oplus \mathbb{R} \oplus \theta
\end{align}
where $\theta = S^n \times \mathbb{R}^{n+1}$ and the involution on
this bundle is simply
\begin{align}
    \mu : S^n \times \mathbb{R}^{n+1} &\to S^n \times \mathbb{R}^{n+1}\\
    (v,y)& \mapsto (-v, y).
\end{align}
We have therefore that ${I}_*=I \oplus \mu$.

\end{itemize}

The involution ${I}_*$ is therefore a conjugation of the stable almost complex structure $ J_{S^n_{\mathrm{st}}}$ of the sphere $S^n$ covering the antipodal map $I$.
The pair $(S^n,{I}_*)$ thus defines an element in the magnetic unitary bordism group of free conjugations:
\begin{align}
    [S^n,{I}_*] \in \mathbf{\Omega}^{\overline{\mathbb{Z}}_2}_n\{ \{1\}\}.
\end{align}

These spheres with stably antipodal conjugation are non-trivial elements in the magnetic unitary bordism group of free conjugations because their images under the forgetful homomorphism
\begin{align}
\mathbf{\Omega}^{\overline{\mathbb{Z}}_2}_n\{ \{1\}\} & \to {\Omega}^{O,{\mathbb{Z}}_2}_n\{ \{1\}\}\\
     [S^n, {I}_*]  & \mapsto [S^n,I],
\end{align}
are non-trivial. Here we have that the canonical embeddings $\mathbb{P}_\mathbb{R}^k \to B\mathbb{Z}_2$ are non-trivial elements in ${\Omega}^{O}_k(B\mathbb{Z}_2) \cong {\Omega}^{O,{\mathbb{Z}}_2}_k\{ \{1\}\}$.

Note furthermore, that this stably structure on the sphere can be seen as the restriction of a stably magnetic structure on the $(n+1)$-dimensional unit ball $B^{n+1}$. The involution on the ball maps $v$ to $-v$ and the conjugation map and the almost complex structure on the stable tangent structure are the following:
\begin{align}
    TB^{n+1} \times \mathbb{R}^{n+1} &\to  TB^{n+1} \times \mathbb{R}^{n+1} \\
    (v, x,y) & \overset{\widetilde{I}}{\mapsto} (-v,-x,y) \\
    (v, x,y) & \overset{J}{\mapsto} (v,-y,x),
\end{align}
for $v \in B^{n+1}$, $x, y \in \mathbb{R}^{n+1}$. Here we assume that $TB^{n+1} = B^{n+1}\times \mathbb{R}^{n+1}$. 

The map of stable vector bundles:
\begin{align}
    TS^n \times \mathbb{R} \times \mathbb{R}^{n+1} &\to TB^{n+1} \times \mathbb{R}^{n+1}\\
    (v,w,r,y)& \mapsto (v,w+rv,y)
\end{align}
is compatible with the almost complex structures and the conjugation maps. Therefore the magnetic unitary manifold with boundary $(B^{n+1}, \widetilde{I})$, which defines an element in $\mathbf{\Omega}^{\overline{\mathbb{Z}}_2}_{n+1}\{\mathcal{A},\{1\}\}$, has the magnetic unitary sphere $(S^{n}, {I}_*)$ as its boundary:
\begin{align}
    \mathbf{\Omega}^{\overline{\mathbb{Z}}_2}_{n+1}\{\mathcal{A},\{1\}\} & \overset{\partial}{\to} \mathbf{\Omega}^{\overline{\mathbb{Z}}_2}_{n}\{\{1\}\} \\
    [B^{n+1}, \widetilde{I}] & \mapsto [S^{n},{I}_*].
\end{align}

Note that the map 
\begin{align}
    TB^{n+1} \times \mathbb{R}^{n+1} &\to  TB^{n+1} \times \mathbb{R}^{n+1} \\
    (v, x, y) & \mapsto (v,x,-y) 
\end{align}
induces a unitary isomorphism between the  almost complex structures $J$ and $-J$ which is compatible with the conjugation map $\widetilde{I}$. So the triples 
$(B^{n+1}, \widetilde{I}, J)$ and
$(B^{n+1}, \widetilde{I}, -J)$ are isomorphic.

\subsection{Smith homomorphism}
A very interesting map in the magnetic unitary bordism groups of free conjugations is the {\it Smith homomorphism}: 
\begin{align}
    \Delta: \mathbf{\Omega}^{\overline{\mathbb{Z}}_2}_{*+1}\{\{1\}\} \to \mathbf{\Omega}^{\overline{\mathbb{Z}}_2}_{*}\{\{1\}\}.
\end{align}
This is the generalization of the Smith homomorphism in the unoriented $\mathbb{Z}_2$-equivariant bordism groups with free involutions defined 
in \cite[\S 26]{Conner_Floyd_Differentiable}.

Allow us describe this map. Take a unitary magnetic manifold $M$ with free conjugation $\tau$. Embed $M$
$\mathbb{Z}_2$-equivariantly into a sphere $S^n$ endowed with the antipodal action, i.e.,  $M\hookrightarrow S^n$. 
Adjust the embedding such that the image of  $M$ is transverse to $S^{n-1}$ in $S^n$ and take the intersection of the image of $M$ with $S^{n-1}$ in $S^n$. Denote this intersection by 
$N \subset M$. Since $S^{n-1}$ is also a $\mathbb{Z}_2$-space, then $\mathbb{Z}_2$ acts freely on $N$. The 
  normal bundle of $N$ in $M$  is trivial. Moreover, the group $\mathbb{Z}_2$ acts on this normal bundle via a fixed one-dimensional representation. Hence $TM|_{N}\cong TN \oplus \mathbb{R}_\pm$ as a $\mathbb{Z}_2$-equivariant real vector bundle, and therefore any magnetic unitary structure on $M$ restricts to one on $N$.

The restriction of the isomorphism $TM \oplus \theta \cong \xi_\mathbb{C}$ to $N$ gives the isomorphism $TN \oplus \mathbb{R}_{\pm} \oplus \theta|_N \cong \xi_\mathbb{C}|_N $ and therefore, $N$ 
is endowed with a stable magnetic unitary structure. The assignment $\Delta(M)=N$ is the Smith homomorphism.

Note that the sphere with antipodal conjugation $(S^{n+1},{I}_* )$ maps to $(S^{n},{I}_*)$ under the Smith homomorphism:
\begin{align}
    \Delta(S^{n+1},{I}_*) = (S^{n},{I}_*),
\end{align}
and iterating this map we can obtain all spheres with antipodal stable conjugation.

These Smith homomorphisms are compatible with the ones in unoriented $\mathbb{Z}_2$-equivariant bordism fitting into the following commutative diagram:
\begin{align}
\xymatrix{
\mathbf{\Omega}^{\overline{\mathbb{Z}}_2}_k\{ \{1\}\} \ar[r]^\Delta \ar[d]& \mathbf{\Omega}^{\overline{\mathbb{Z}}_2}_{k-1}\{ \{1\}\} \ar[r]^\Delta \ar[d] & \cdots \ar[r]^\Delta &\mathbf{\Omega}^{\overline{\mathbb{Z}}_2}_{0}\{ \{1\}\} \ar[d]\\
    {\Omega}^{O,{\mathbb{Z}}_2}_k\{ \{1\}\} \ar[r]^\Delta & {\Omega}^{O,{\mathbb{Z}}_2}_{k-1}\{ \{1\}\} \ar[r]^\Delta & \cdots \ar[r]^\Delta &  {\Omega}^{O,{\mathbb{Z}}_2}_0\{ \{1\}\}
    }
\end{align}
where the vertical maps forget the unitary structure and keep the conjugations as involutions.

Having described the Smith homomorphism, we are ready to present an exact sequence that incorporates both the Milnor map and the Smith homomorphism. This has been developed by Stong \cite{Stong_manifolds_with_conjugation} in the context of AR bordism theory.

\subsection{Exact sequence for the Milnor map}
It turns out that the cokernel of the Smith homomorphism is isomorphic to the kernel of the Milnor map. 
This fact allows us to introduce the following exact sequence incorporating both the Smith homomorphism and the Milnor map.

\begin{theorem} \label{theorem_sequences_Milnor_map}
    The sequences of homomorphisms:
    \begin{align}
      \mathbf{\Omega}^{\overline{\mathbb{Z}}_2}_{2*+1}\{ \{1\}\} \overset{\Delta}{\longrightarrow} \mathbf{\Omega}^{\overline{\mathbb{Z}}_2}_{2*}\{ \{1\}\}  \overset{F_U \circ \iota}{\longrightarrow} \Omega^{U}_{2*} \overset{\mu}{\longrightarrow} \Omega^{O}_* & \longrightarrow 0 \\
      \mathbf{\Omega}^{\overline{\mathbb{Z}}_2}_{2*}\{ \{1\}\} \overset{\Delta}{\longrightarrow} \mathbf{\Omega}^{\overline{\mathbb{Z}}_2}_{2*-1}\{ \{1\}\}  \longrightarrow 0 &
    \end{align}
    are exact.
\end{theorem}
\begin{proof}
    Let us start showing that the image of the Smith homomorphism produces unitary manifolds which bound; that is: $\mathrm{im}(\Delta) \subset\mathrm{ker}(F_U \circ \iota)$.
Take any magnetic unitary manifold $M$ with free conjugation and let $N=\Delta(M)$ be the image of the Smith homomorphism. By construction of the Smith homomorphism the manifold $M$ can be split into two manifolds with boundary $M_+$ and $M_-$, with $N= \partial M_-=\partial M_+$ and with the conjugation involution interchanging $M_+$ and $M_-$. The manifold with boundary $M_+$ inherits a unitary structure from $M$ and bounds $N$. Hence we have proven:
 \begin{align}
 \mathrm{im}(\Delta) \subset\mathrm{ker}(F_U \circ \iota).
 \end{align}

Take now any magnetic unitary manifold $N$ with free conjugation that bounds as a unitary manifold $\partial W=N$, and let us show that $N$ is in the image of the Smith homomorphism; that is:
$\mathrm{ker}(F_U \circ \iota) \subset \mathrm{im}(\Delta)$.
Denote by $\tau:N \to N$ the conjugation of $N$ and by $J$ the stable almost complex structure of $W$. Consider the unitary manifolds $(W \times \{-1\},-J)$ and $(W \times \{1\}),J)$ with opposite stable almost complex structures and 
consider the map:
\begin{align}
 \phi:   N \times \{-1\} &\to N \times \{1\}\\
    (n,-1)  & \mapsto (\tau(n),1).
\end{align}
Since on the stable tangent bundle of $N$ we have the equation $\tau_* \circ J|_N = -J|_N \circ \tau_*$, we see that the gluing 
\begin{align} \label{gluing 2W makes M}
M= (W \times \{-1\} )\underset{\phi}{\sqcup} ( W\times \{1\})
\end{align}
possesses a stable almost complex structure $J_M$ satisfying the equations:
\begin{align} \label{equations_of_J_M}
(J_M)|_{W\times\{1\}}=J \ \ \mathrm{and}\  \  (J_M)|_{W\times\{-1\}}=-J.
\end{align}The involution:
\begin{align} \label{involution_tilde-tau}
    \tilde{\tau}: M & \to M\\
    (w,\pm 1) & \mapsto (w, \mp 1)
\end{align}
restricts to the involution $\tau$ on $N$, and the equation $\tilde{\tau}_* \circ J_M = -J_M \circ \tilde{\tau}_*$
follows from Eqns. \eqref{equations_of_J_M}. Therefore $M$ is a magnetic unitary manifold with free conjugation.

Embed equivariantly $N$ to some sphere $S^L$, and extend this embedding to an equivariant embedding of $M$ into $S^{L+1}$ such that $W \times \{1\}$ lies in the northern hemisphere $S^{L+1}_+$. With this embedding of $M$ it is clear that the image of $M$ under the Smith homomorphism is $N$. Therefore we have that:
\begin{align}
    \mathrm{ker}(F_U \circ \iota) \subset \mathrm{im}(\Delta).
\end{align}

We are just left to prove that the kernel of the Milnor map agrees with the image of the forgetful functor from magnetic unitary bordism of free conjugations to unitary bordism.

First note that for any magnetic unitary manifold endowed with a free involution, its underlying unitary manifold lies in the kernel of the Milnor map. Hence we have:
\begin{align}
    \mathrm{im}(F_U \circ \iota) \subset \mathrm{ker}(\mu).
\end{align}

On the other hand, note that we have explicitly calculated the kernel of the Milnor map $\mu$ in Lem. \ref{lemma_kernel_mu}. This kernel is generated by $2\Omega^{U}_{2*}$ and the Milnor manifolds $H_\mathbb{C}(2^s,2^s)$ for $s\geq 0$. The twin magnetic unitary manifolds of Eqn. \eqref{twin_magnetic_unitary_manifolds}
generate $2\Omega^{U}_{2*}$ once the conjugation is forgotten, and the Milnor hypersurfaces $H_\mathbb{C}(2^s,2^s)$ can be endowed with a free conjugation as described in Eqn. \eqref{Milnor_hypersurfaces_free}. Therefore, the Milnor hypersurfaces
$H_\mathbb{C}(2^s,2^s)$ are in the image of the forgetful map $F_U \circ \iota$. Further, note that one can take the product of the Milnor hypersurfaces with any unitary manifold with conjugation and the resulting manifold is also a free conjugation unitary manifold. Thus, we have just shown that:
\begin{align}
       \mathrm{ker}(\mu) \subset \mathrm{im}(F_U \circ \iota).
\end{align}

\end{proof}

We now present some diagrams involving the previously defined magnetic unitary manifolds. With respect to the unrestricted-restricted long exact sequence we have the assignments:
\begin{align}
\xymatrixrowsep{0mm}
\xymatrix{\mathbf{\Omega}^{\overline{\mathbb{Z}}_2}_{*+1}\{\mathcal{A}, \{1\}\} \ar[r]^\partial & 
\mathbf{\Omega}^{\overline{\mathbb{Z}}_2}_*\{\{1\}\} \ar[r]^\iota &
\mathbf{\Omega}^{\overline{\mathbb{Z}}_2}_* \ar[r]^\Psi &
\mathbf{\Omega}^{\overline{\mathbb{Z}}_2}_*\{\mathcal{A}, \{1\}\} \\
[B^{2n}, \widetilde{I}] \ar@{|->}[r] & [S^{2n-1}, {I}_*] \ar@{|->}[r] & 0 & \\
 & [M\sqcup M,  (\mathbb{K} \times(-1))_*]  \ar@{|->}[r] &  [M\sqcup M,  (\mathbb{K} \times(-1))_*]  \ar@{|->}[r] & 0 \\
& [H_\mathbb{C}(2^s,2^s) , {\mathbb{K}_{\mathrm{swap}}}_*]  \ar@{|->}[r]& [H_\mathbb{C}(2^s,2^s) ,{\mathbb{K}_{\mathrm{swap}}}_*]  \ar@{|->}[r] & 0.
  }
\end{align}

With respect to the exact sequence of Thm. \ref{theorem_sequences_Milnor_map} we have the assignments:
\begin{align}
\xymatrixrowsep{0mm}
\xymatrix{
\mathbf{\Omega}^{\overline{\mathbb{Z}}_2}_{*+1}\{ \{1\}\} \ar[r]^\Delta & \mathbf{\Omega}^{\overline{\mathbb{Z}}_2}_{*}\{ \{1\}\}  \ar[r]^{F_U \circ \iota} & \Omega^{U}_{*} \ar[r]^\mu & \Omega^{O}_*   \\
[S^{n+1}, {I}_*] \ar@{|->}[r] & [S^{n}, {I}_*]  \ar@{|->}[r] & 0 & \\
& [M\sqcup M,  (\mathbb{K} \times(-1))_*]  \ar@{|->}[r] &  2[M]  \ar@{|->}[r] & 0\\
 &  [H_\mathbb{C}(2^s,2^s) , {\mathbb{K}_{\mathrm{swap}}}_*]  \ar@{|->}[r]& [H_\mathbb{C}(2^s,2^s)] \ar@{|->}[r]& 0 \\
 && [H_{\mathbb{C}}(2^k,2t2^k)] \ar@{|->}[r]&  [H_{\mathbb{R}}(2^k,2t2^k)]\\
 &&
[\mathbb{P}_\mathbb{C}^{2r}]  \ar@{|->}[r]& [\mathbb{P}_\mathbb{R}^{2r}] 
 }     
\end{align}

Note that the magnetic unitary manifolds with free conjugation 
$2(H_\mathbb{C}(2^s,2^s) , {\mathbb{K}_{\mathrm{swap}}}_*)$ and 
$(H_\mathbb{C}(2^s,2^s) \sqcup H_\mathbb{C}(2^s,2^s), (\mathbb{K}\times (-1))_*)$ map under $F_U \circ \iota$ to the same unitary manifold. So, in the case that they do not belong to the same bordism class in $\mathbf{\Omega}^{\overline{\mathbb{Z}}_2}_{2(2^{s+1}-1)}\{ \{1\}\}$, there must exist a magnetic unitary manifold with free conjugation $W$ such that
\begin{align}
    \Delta(W) = 2(H_\mathbb{C}(2^s,2^s) , {\mathbb{K}_{\mathrm{swap}}}_*) - (H_\mathbb{C}(2^s,2^s) \sqcup H_\mathbb{C}(2^s,2^s), (\mathbb{K}\times (-1))_*).
\end{align}

Let us consider $0$-dimensional magnetic bordism $\mathbf{\Omega}^{\overline{\mathbb{Z}}_2}_{0}$. Note that a point with $\mathbb{C}$ as stably magnetic bundle generate a copy of $\mathbb{Z}$ in $\mathbf{\Omega}^{\overline{\mathbb{Z}}_2}_{0}$. If we consider instead two points $\{-1,1\}$ with free $\mathbb{Z}_2$-action and we take the class of $\mathbb{C}_1 \sqcup \overline{\mathbb{C}}_{-1}$ for its stably complex vector bundle, we see that the involution $\nu:\mathbb{C}_1 \to \overline{\mathbb{C}}_{-1}$, $\nu(z)= z$ is a conjugation. This magnetic unitary manifold $\mathbb{C}_1 \sqcup \overline{\mathbb{C}}_{-1}$ generates a $2$-torsion class in $\mathbf{\Omega}^{\overline{\mathbb{Z}}_2}_{0}\{\{1\}\}$ which maps to zero in $\mathbf{\Omega}^{\overline{\mathbb{Z}}_2}_{0}$. In this case we have:
\begin{align}
\mathbf{\Omega}^{\overline{\mathbb{Z}}_2}_{0} & \cong \langle [\mathbb{C} , \mathbb{K}_*]\rangle  \cong \mathbb{Z},\\
\mathbf{\Omega}^{\overline{\mathbb{Z}}_2}_{0} \{\{1\}\} & \cong \langle [\mathbb{C} \sqcup \mathbb{C}, {\mathbb{K}_{\mathrm{swap}}}_*], [\mathbb{C} \sqcup \overline{\mathbb{C}},\nu] \rangle \cong \mathbb{Z} \oplus \mathbb{Z}_2.
\end{align}

\subsection{Exact sequence for the $\mathbb{Z}_2$-equivariant Milnor map}

The generalization of the sequences of Thm. \ref{theorem_sequences_Milnor_map} to the magnetic group 
$\mathbb{Z}_2 \times \overline{\mathbb{Z}}_2$ is 
possible due to our calculation in Lem. \ref{lemma_kernel_muZ2} of 
the kernel of the $\mathbb{Z}_2$-equivariant Milnor map 
$\mu^{\mathbb{Z}_2}$.
Here the magnetic group $\mathbb{Z}_2 \times \overline{\mathbb{Z}}_2$ 
means the group $\mathbb{Z}_2 \times {\mathbb{Z}}_2$ with $\phi$ the projection in the second 
$\mathbb{Z}_2$.

Let $\mathbf{\Omega}^{\mathbb{Z}_2\times\overline{\mathbb{Z}}_2}_{*}$ 
be the magnetic $\mathbb{Z}_2\times\overline{\mathbb{Z}}_2$-equivariant bordism groups, and consider the bordism groups where the conjugation (namely the action of the group $\overline{\mathbb{Z}}_2$) is free. For this we simply consider the family of subgroups $\mathcal{F} = \{ \langle (1,0)\rangle, \langle (1,1)\rangle, \{(0,0)\} \}$
which does not contain the subgroup generated by $(0,1)$. The bordism groups $\mathbf{\Omega}^{\mathbb{Z}_2\times\overline{\mathbb{Z}}_2}_{*}\{ \mathcal{F}\}$ of magnetic $\mathbb{Z}_2\times\overline{\mathbb{Z}}_2$-equivariant manifolds whose isotropy groups restrict to $\mathcal{F}$ are precisely the equivariant bordism groups of free conjugations.

In terms of spaces we could take a model for the space $E \mathcal{F}$. Consider $V_+$ and $V_-$ the real irreducible representations of $\mathbb{Z}_2 \times \{0\}$ and consider the spheres $S((V_+\times V_-)^{n})$ where the group $\{0\} \times \overline{\mathbb{Z}}_2$ acts by the antipodal map (multiplication by $-1$). Taking the limit when $n$
 goes to infinity we end up with an infinite sphere $S((V_+\times V_-)^{\infty})$ with an action of the group $\mathbb{Z}_2\times\overline{\mathbb{Z}}_2$. This space is contractible, and the fixed point sets of the elements $(1,0)$ and $(1,1)$ are also contractible, being respectively the infinite spheres $S((V_+)^{ \infty})$ and
$S((V_-)^{ \infty})$. Since the action of $\{0\} \times \overline{\mathbb{Z}}_2$ is free, we see that a model for $E \mathcal{F}$ is this infinite sphere $S((V_+\times V_-)^{ \infty})$. Let us then take
$E \mathcal{F} = S((V_+\times V_-)^{ \infty})$, and thus we have the isomorphisms:
\begin{align}    
\mathbf{\Omega}^{\mathbb{Z}_2\times\overline{\mathbb{Z}}_2}_{*}\{ \mathcal{F}\} \cong \mathbf{\Omega}^{\mathbb{Z}_2\times\overline{\mathbb{Z}}_2}_{*}(E \mathcal{F}) =  \mathbf{\Omega}^{\mathbb{Z}_2\times\overline{\mathbb{Z}}_2}_{*}(S((V_+\times V_-)^{ \infty})).
\end{align}

There are two Smith homomorphisms that can be defined in this setup:
\begin{align} \overline{\Delta}_\pm:   
\mathbf{\Omega}^{\mathbb{Z}_2\times\overline{\mathbb{Z}}_2}_{*+1}\{ \mathcal{F}\} {\to}
\mathbf{\Omega}^{\mathbb{Z}_2\times\overline{\mathbb{Z}}_2}_{*}\{ \mathcal{F}\}.
\end{align}
Take a magnetic $\mathbb{Z}_2\times\overline{\mathbb{Z}}_2$-equivariant manifold $M$ with free conjugation in $\mathbf{\Omega}^{\mathbb{Z}_2\times\overline{\mathbb{Z}}_2}_{*+1}\{ \mathcal{F}\}$, and consider its equivariant classifying map $M \to E\mathcal{F}$.
 This map factors through a finite-dimensional sphere $f: M \to S((V_+ \times V_-)^{ n})$ for some sufficiently large integer $n$
and we want to find its transversal intersection with a $\mathbb{Z}_2\times \{0\}$-invariant sphere $S^{2n-2}$. There are two distinctly different types of 
$\mathbb{Z}_2\times \{0\}$-invariant $(2n-2)$-dimensional spheres in $S((V_+ \times V_-)^{ n})$. We take either the sphere $S(V_+^{n-1} \times V_-^n)$ or the sphere $S(V_+^{n} \times V_-^{n-1})$.

The transversal intersection of $M$ with the sphere $S(V_+^{n} \times V_-^{n-1})$ will be denoted $N_+$ while the transversal intersection with the sphere $S(V_+^{n-1} \times V_-^n)$ will be denoted $N_-$. The Smith homomorphisms are thus defined by the assignments:
\begin{align}
    \overline{\Delta}_+(M) = N_+ \ \ \mathrm{and}  \ \ \overline{\Delta}_-(M) = N_-.
\end{align}

Consider the natural maps:
\begin{align}
\mathbf{\Omega}^{\mathbb{Z}_2\times\overline{\mathbb{Z}}_2}_{*}\{ \mathcal{F}\}
\overset{\overline{\iota}}{\longrightarrow} \mathbf{\Omega}^{\mathbb{Z}_2\times\overline{\mathbb{Z}}_2}_{*}  \overset{\overline{F}_U}{\longrightarrow} \Omega^{U, \mathbb{Z}_2}_*,
\end{align}
where the first map sends restricted bordism to unrestricted bordism, and the second forgets the conjugation. We allow ourselves to claim the following generalization of Thm. \ref{theorem_sequences_Milnor_map}.

\begin{theorem} \label{theorem_sequence_Z2_Milnor_map}
    The sequence of homomorphisms:
    \begin{align}
\mathbf{\Omega}^{\mathbb{Z}_2\times \overline{\mathbb{Z}}_2}_{*+1}\{ \mathcal{F}\} \overset{\overline{\Delta}_++\overline{\Delta}_-}{\longrightarrow} \mathbf{\Omega}^{\mathbb{Z}_2\times \overline{\mathbb{Z}}_2}_{*}\{ \mathcal{F}\}  \overset{\overline{F}_U \circ \overline{\iota}}{\longrightarrow} \Omega^{U, \mathbb{Z}_2}_{2*} \overset{\mu^{\mathbb{Z}_2}}{\longrightarrow} \Omega^{O, \mathbb{Z}_2}_* & \longrightarrow 0 
    \end{align}
are exact. Here $\mu^{\mathbb{Z}_2}$ is the equivariant Milnor map of Eqn. \eqref{equivariant_Milnor-Map} and $\overline{\Delta}_++\overline{\Delta}_-$ denotes the span of the images of the Smith homomorphisms.
\end{theorem}

\begin{proof}
    Most of the arguments  in the proof of Thm. \ref{theorem_sequences_Milnor_map}
    can be extrapolated to this equivariant setup. Hence we will be concise.

    Any equivalence class in $\Omega^{U, \mathbb{Z}_2}_{2*}$ is realized by a smooth complex algebraic variety $M$ with $\mathbb{Z}_2$-action.
    The complex conjugation map $\mathbb{K}$ is the conjugation that makes this manifold magnetic equivariant. 
    Take the twin manifold $M \sqcup M$ with their $\mathbb{Z}_2$-action  and endow the pair with the conjugation $\mathbb{K}\times (-1)$ as in Eqn. \eqref{twin-magnetic-mainfold}. 
    The twin manifold $(M \sqcup M,\mathbb{K}\times (-1)) $ defined an element in $\mathbf{\Omega}^{\mathbb{Z}_2\times \overline{\mathbb{Z}}_2}_{*}\{ \mathcal{F}\}$ which maps to $2M$ under the map $\overline{F}_U \circ \overline{\iota}$. This means that $2\Omega^{U, \mathbb{Z}_2}_{2*} \subset \mathrm{im}(\overline{F}_U \circ \overline{\iota})$.

    The Milnor Hypersurfaces $(H_\mathbb{C}(2^s,2^s), \mathbb{K}_{\mathrm{swap}})$ with  free conjugation map and trivial $\mathbb{Z}_2$-action define elements in $\mathbf{\Omega}^{\mathbb{Z}_2\times \overline{\mathbb{Z}}_2}_{*}\{ \mathcal{F}\}$. Therefore the ideal generated by these Milnor hypersurfaces are also in the image of $\overline{F}_U \circ \overline{\iota}$. This means 
    $\Omega^{U, \mathbb{Z}_2}_{2*} \langle H_\mathbb{C}(2^s,2^s)\rangle\subset \mathrm{im}(\overline{F}_U \circ \overline{\iota})$, and by Lem. \ref{lemma_kernel_muZ2} we see that $\mathrm{ker}(\mu^{\mathbb{Z}_2}) \subset \mathrm{im}(\overline{F}_U \circ \overline{\iota})$.

   Now, any manifold with free conjugation maps to zero by the equivariant Milnor map. This implies that $ \mathrm{im}(\overline{F}_U \circ \overline{\iota}) \subset \mathrm{ker}(\mu^{\mathbb{Z}_2})$ and therefore we obtain the desired equation:
   \begin{align}
       \mathrm{im}(\overline{F}_U \circ \overline{\iota}) = \mathrm{ker}(\mu^{\mathbb{Z}_2}).
   \end{align}

The images of the Smith homomorphisms $\overline{\Delta}_\pm$ are in the kernel of the forgetful map $\overline{F}_U \circ \overline{\iota}$.
By construction, if $\overline{\Delta}_\pm(M)=N_\pm$, then
the manifold $M$ can be split into two copies and each one has for boundary the manifold $N_\pm$. Hence $N_\pm$ bounds in $\Omega^{U,\mathbb{Z}_2}_{2*}$. We have then $\mathrm{im}(\overline{\Delta}_\pm) \subset \mathrm{ker}(\overline{F}_U \circ \overline{\iota})$. 

Let us now suppose that $N$ is a manifold in $\mathbf{\Omega}^{\mathbb{Z}_2\times \overline{\mathbb{Z}}_2}_{*}\{ \mathcal{F}\} $ which bounds in $\Omega^{U,\mathbb{Z}_2}_{2*}$. Let $W$ be the $\mathbb{Z}_2$-equivariant unitary manifold which bounds $N$ with $\partial W=N$ and $J_W|_N=J_N$. Define 
$M$ in the same form as in Eqn. \eqref{gluing 2W makes M} and endowed with the same involution $\tilde{\tau}$
as defined in Eqn. \eqref{involution_tilde-tau}. Now we just need to check that the original action of $\mathbb{Z}_2$ in $W$ is compatible with the action of the conjugation $\tilde{\tau}$. But this follows from the definition of $\tilde{\tau}$ in Eqn \eqref{involution_tilde-tau} since we have the equations:
\begin{align}
    \tilde{\tau}(gw, \pm 1) = (gw, \mp 1) = g(\tilde{\tau}(w, \pm 1)),
\end{align}
where $g$ is the generator of the group $\mathbb{Z}_2$.

Therefore $M$ is a magnetic $\mathbb{Z}_2 \times \overline{\mathbb{Z}}_2$-equivariant manifold. Now, find a $\mathbb{Z}_2$-equivariant embedding of $W$ into a representation $(V_+\times V_-)^n$ for some large $n$, such that $\partial W=N$
lands in the codimension 1-subspace
$V_+^{n-1}\times V_-^n$ and moreover, that the map from $N$ to this space 
is $\mathbb{Z}_2 \times \overline{\mathbb{Z}}_2$-equivariant.
Here the action of $\overline{\mathbb{Z}}_2$ in $V_+^{n-1}\times V_-^n$  is via the sign representation.

Assume that $W \cap (V_+^{n-1}\times V_-^n)=N$ and therefore $W$ lies in the upper half space of $(V_+\times V_-)^n$.

Since $0$ is not in the image of $W$ in the vector space $V_+^{n-1}\times V_-^n$, we can use a homotopy in $(V_+ \times V_-)^n \backslash \{0\}$ to obtain a $\mathbb{Z}_2 \times \overline{\mathbb{Z}}_2$-equivariant map from $W$ to the sphere $S((V_+\times V_-)^n)$ such that $N$ is the intersection of that map with the sphere  $S(V_+^{n-1}\times V_-^n)$.
Therefore $N$ is in the image of the Smith homomorphism $\overline{\Delta}_+$.

If instead we find the map of $N$ to be in the codimension 1 representation $V_+^n\times V_-^{n-1}$, then  $N$ would lie in the image of the Smith homomorphism $\overline{\Delta}_-$.

We conclude with the equation:
\begin{align}
    \mathrm{im}(\overline{\Delta}_+ + \overline{\Delta}_-) = \mathrm{ker}(\overline{F}_U \circ \overline{\iota}).
\end{align}

\end{proof}

\section*{Acknowledgments}

Zhi L\"u thanks the Shanghai Institute for Mathematics and Interdisciplinary Sciences (SIMIS) for its financial support. This research was partially  funded by SIMIS under grant number SIMIS-ID-2025-TP, and  by  NSFC through grant number 11971112. Bernardo Uribe acknowledges the  financial support of the Institute of the Mathematical Sciences of the Americas in Miami, USA, of the Max Planck Institute for Mathematics in Bonn, Germany, of the International Center for Theoretical Physics in Trieste, Italy, through its associates program, of the Alexander Von Humboldt Foundation in Bonn, Germany, and of Fudan University in Shanghai, China. 

\section*{Tools and computational resources disclosure}
There was no use of any mathematical assistance AI in the elaboration of this work.
All statements and proofs were devised by the authors themselves. Only grammar and spelling correction from the latex console was employed to ensure proper use of the English language.

\bibliographystyle{alpha}
\bibliography{ref}

\end{document}